\documentclass[twocolumn,twoside]{IEEEtran}
\usepackage{amsmath,amssymb,amsthm,epsfig,dsfont,color,subfigure,empheq,graphicx}
\usepackage{enumerate,url,algpseudocode,algorithm,wasysym,epstopdf,verbatim,array,textcomp,balance}
\usepackage[algo2e]{algorithm2e}
\usepackage[table]{xcolor}

\DeclareMathOperator{\sign}{sign}
\DeclareMathOperator{\nullspace}{null}

\newtheorem{condition}{Condition}

\newtheorem{lemma}{Lemma}
\newtheorem{corollary}{Corollary}
\newtheorem{theorem}{Theorem}
\newtheorem{definition}{Definition}
\newtheorem{remark}{Remark}

\newcommand \bzero{\mathbf{0}}
\newcommand \bone{\mathbf{1}}
\newcommand \ba{\mathbf{a}}

\newcommand \bg{\mathbf{g}} 

\newcommand \bn{\mathbf{n}}

\newcommand \bq{\mathbf{q}}

\newcommand \bx{\mathbf{x}}
\newcommand \by{\mathbf{y}}

\newcommand \bA{\mathbf{A}}

\newcommand \bJ{\mathbf{J}}

\newcommand \bxi{\boldsymbol{\xi}}
\newcommand \bpi{\boldsymbol{\pi}}

\newcommand \bphi{\boldsymbol{\phi}}

\newcommand \bpsi{\boldsymbol{\psi}}

\newcommand \htphi{\hat{\phi}}

\newcommand \htpsi{\hat{\psi}}

\newcommand \mcC{\mathcal{C}}

\newcommand \mcG{\mathcal{G}}

\newcommand \mcN{\mathcal{N}}

\newcommand \mcP{\mathcal{P}}

\newcommand \mcS{\mathcal{S}}
\newcommand \mcT{\mathcal{T}}

\newcommand \bmcP{\bar{\mathcal{P}}}

\newcommand \tbq{\tilde{\mathbf{q}}}

\newcommand \tbphi{\tilde{\boldsymbol{\phi}}}

\newcommand \hbphi{\hat{\boldsymbol{\phi}}}

\newcommand \hbpsi{\hat{\boldsymbol{\psi}}}

\begin{document}
\title{Natural Gas Flow Solvers using Convex Relaxation}

\author{Manish K. Singh~\IEEEmembership{Student Member,~IEEE} and Vassilis Kekatos,~\IEEEmembership{Senior Member,~IEEE}

\thanks{Manuscript received on June 3, 2019; revised on November 20, 2019; accepted on January 21, 2020. Date of publication DATE; date of current version DATE. Paper no. TCONES-19-0184.}

\thanks{The authors are with the Bradley Dept. of ECE, Virginia Tech, Blacksburg, VA 24061, USA. Emails: \{manishks,kekatos\}@vt.edu. This research was supported by the U.S. National Science Foundation under Grant 1711587.}

\thanks{Color versions of one or more of the figures is this paper are available online at {http://ieeexplore.ieee.org}.}
\thanks{Digital Object Identifier XXXXXX}
}

\markboth{IEEE TRANSACTIONS ON CONTROL OF NETWORK SYSTEMS (to appear)}{Singh and Kekatos: Natural Gas Flow Solvers using Convex Relaxation}

\maketitle

\begin{abstract}
The vast infrastructure development, gas flow dynamics, and complex interdependence of gas with electric power networks call for advanced computational tools. Solving the equations relating gas injections to pressures and pipeline flows lies at the heart of natural gas network (NGN) operation, yet existing solvers require careful initialization and uniqueness has been an open question. In this context, this work considers the nonlinear steady-state version of the gas flow (GF) problem. It first establishes that the solution to the GF problem is unique under arbitrary NGN topologies, compressor types, and sets of specifications. For GF setups where pressure is specified on a single (reference) node and compressors do no appear in cycles, the GF task is posed as an convex minimization. To handle more general setups, a GF solver relying on a mixed-integer quadratically-constrained quadratic program (MI-QCQP) is also devised. This solver can be used for any GF setup at any NGN. It introduces binary variables to capture flow directions; relaxes the pressure drop equations to quadratic inequality constraints; and uses a carefully selected objective to promote the exactness of this relaxation. The relaxation is provably exact in NGNs with non-overlapping cycles and a single fixed-pressure node. The solver handles efficiently the involved bilinear terms through McCormick linearization. Numerical tests validate our claims, demonstrate that the MI-QCQP solver scales well, and that the relaxation is exact even when the sufficient conditions are violated, such as in NGNs with overlapping cycles and multiple fixed-pressure nodes. 
\end{abstract}


\section{Introduction}
\allowdisplaybreaks
Natural gas has served as a critical energy source for decades, mainly for heating and electric power generation~\cite{mercado2015review}. Thanks to the higher ramping capabilities of gas-fired generators, electric power system operators could achieve higher penetration of uncertain and intermittent renewable generation. In addition, the discovery of substantial new supplies of natural gas in the U.S. has led to a new thrust in development of gas-centered technologies and analytical tools~\cite{MIT}. 

Natural gas produced at gas pits and refineries is primarily transported to customer locations via a continent-wide network of pipelines~\cite{mercado2015review}. The safe, reliable, and economical transportation of gas across these networks is ensured by gas system operators~\cite{Zlotnik2017coordinated}. Considering the scale of natural gas networks (NGN), and their coupling with electric power grids, a plethora of analytical and computational challenges can be envisaged. Stand-alone and gas-electric coupled versions of network expansion planning, optimal scheduling, least-cost procurement, and security analysis are examples of problems that have gained increasing research interest; see~\cite{Zlotnik2017coordinated}, \cite{backhaus2016convex}, \cite{bent2018market}, \cite{Zlotnik16b}, \cite{Schwele2019coordination}. These problems aim at optimizing varying objectives, while respecting network limitations and gas flow physics. 

The flow of natural gas on pipelines is governed by partial differential equations, which under steady-state assumptions, yield nonlinear equations relating pressures and gas flows~\cite{Singh18ACC}. These equations reveal that the pressure drops along a pipe in the direction of flow due to friction. However, a minimum pressure needs to be maintained at consumer nodes to satisfy gas contracts. Therefore, compressors are placed on selected pipelines to increase the pressure at their output based on a typically multiplicative~\cite{mercado2015review}, and rarely additive law~\cite{Vuffray2015monotonicity}. Operators need to solve the set of nonlinear equations governing gas flow in an NGN~\cite{abhi2017GM}: For each node, the operator fixes the gas pressure or gas injection rate to specified values. Given also the compression ratios, the GF task aims at finding the injections and pressures at all nodes, as well as the gas flows on all pipes. While solving the GF task is central for numerous NGN operations, it is hard to do so even under steady-state and balanced conditions for non-tree networks~\cite{mercado2015review}. 

The GF task is usually handled by Newton-Raphson (NR)-based solvers. However, their convergence can be sensitive to initialization~\cite{esquivel2012unified}. A semidefinite program (SDP)-based GF solver attaining a higher success probability than the NR scheme, is developed in \cite{abhi2017GM}. Nevertheless, the SDP based solver fails to solve the GF problem if the network state is far from the states considered in designing the solver. The necessity of proper initialization may be avoided for simpler networks without compressors as the flows and pressures may be found as optimal primal-dual solutions of a convex minimization~\cite{wolf2000energy}. Nevertheless, for practical meshed NGNs with compressors, an initialization-independent GF solver is still a research pursuit~\cite{Singh18ACC}. Setting scalability aside, if one uses a nonlinear solver for the GF task, the uniqueness of a solution becomes critical. References \cite{Vuffray2015monotonicity} and \cite{misra2015maximum} prove the uniqueness of a GF solution for NGNs with \emph{additive} compressors.


The contribution of this work is on four fronts: First, Section~\ref{sec:unique} establishes that the nonlinear steady-state GF equations enjoy a unique solution even with multiplicative compressors. Building on \cite{Singh18ACC} where uniqueness was shown for GF setups with a single fixed-pressure node, here uniqueness is non-trivially generalized to setups with multiple fixed-pressure nodes. Second, Section~\ref{sec:efm} reformulates the GF task as a convex minimization. The obtained solver can handle GF setups with a single fixed-pressure node and compressors not on cycles. Third, Section~\ref{sec:cr} expands the analytical claims for the MI-QCQP gas flow solver of \cite{Singh18ACC}. Different from the convex minimization approach, this solver applies to any GF setup and any network. The MI-QCQP solver introduces binary variables to capture flow directions; relaxes the nonlinear GF equations to quadratic inequalities; and uses a carefully selected objective to promote the exactness of the relaxation. The relaxation is provably exact in NGNs with non-overlapping cycles and a single fixed-pressure node. This significantly extends the claim of \cite{Singh18ACC}, where exactness was proved for non-overlapping cycles and a single fixed-pressure node, but did not allow for compressors in cycles. Having compressors in cycles is a typical arrangement, e.g., when two compressors are connected in parallel. Fourth, to accelerate the MI-QCQP solver, the bilinear terms involved are handled through McCormick linearization. Numerical tests on meshed networks with overlapping cycles and multiple fixed-pressure nodes demonstrate that the MI-QCQP solver finds the unique GF solution even when the assumed sufficient conditions are violated.  

\section{Gas Flow Problem}\label{sec:model}
A natural gas network (NGN) can be represented by a directed graph $\mcG=(\mcN,\mcP)$. The nodes in the graph represent points of gas supply, demand, or network junctions. The edges are \emph{directed}, and represent pipelines or compressors. Nodes are indexed by $n\in\mcN:=\{1,\cdots,N\}$ and edges by $\ell\in\mcP:=\{1,\ldots,P\}$. Each edge $\ell=(m,n)$ is assigned a direction from the origin node $m$ to the destination node $n$. If $(m,n)\in\mcP$, then $(n,m)\notin\mcP$. For edges corresponding to pipes, this direction is selected arbitrarily. For edges denoting compressors, the direction coincides with the direction of gas flow, since compressors allow only unidirectional flow of gas. 

For each node $n\in\mcN$, let $q_n$ be the gas injection rate from node $n$ to the NGN. By convention, the gas injection $q_n$ is positive for gas source nodes; negative for demand nodes; and zero for junction nodes. Vector $\bq\in\mathbb{R}^N$ collects the gas injections across all nodes.

For each edge $\ell=(m,n)\in\mcP$, let $\phi_\ell$ denote its gas flow rate. By convention, the flow $\phi_\ell$ is positive if gas flows from node $m$ to $n$; and negative, otherwise. The conservation of mass at each node $n\in\mcN$ dictates that
\begin{equation}\label{eq:mc}
q_n = \sum_{\ell:(n,k) \in \mcP}\phi_\ell - \sum_{\ell:(k,n) \in \mcP}\phi_\ell.
\end{equation}
Under steady-state conditions, the input and output flows on a pipe are identical, and so gas injections are balanced at all times, that is $\sum_{n=1}^{N}q_n=0$. Because of this, from the $N$ linear equations in \eqref{eq:mc}, only $(N-1)$ are linearly independent. 

The topology of the NGN is captured by its edge-node incidence matrix $\bA\in\mathbb{R}^{P\times N}$ with entries
\begin{equation*}
A_{\ell,k}:=
\begin{cases}
+1&,~k=m\\
-1&,~k=n\\
0&,~\text{otherwise}
\end{cases}~\forall~\ell=(m,n)\in\mcP.
\end{equation*}
Using $\bA$, equation \eqref{eq:mc} can be compactly expressed as
\begin{equation}\label{eq:mc2}
\bA^\top\bphi=\bq.
\end{equation}
where vector $\bphi\in\mathbb{R}^P$ stacks the flows $\phi_\ell$'s along all edges.

For medium- and high-pressure networks, the gas flows on pipelines relate to nodal pressures through a set of nonlinear partial differential equations~\cite{Osiadacz87}, \cite{ThorleyTiley87}. These equations model the gas flow dynamics evolving across time and spatially along the pipeline length. However, simplifying assumptions such as ignoring friction, geographical tilt, variations in ambient temperature, and time-varying gas injections, yield the popular steady-state Weymouth equation~\cite{Wu99}. If $\psi_n>0$ denotes the \emph{squared} gas pressure at node $n\in\mcN$, the pressure drop across pipeline $\ell=(m,n)\in\mcP$ is given by
\begin{subequations}\label{eq:wey2}
	\begin{align}	\psi_m-\psi_n&=a_\ell\sign(\phi_\ell)\phi_\ell^2\label{seq:wey2a}\\
	\psi_n&\geq0\label{seq:wey2b}
	\end{align}
\end{subequations}
where parameter $a_\ell>0$ depends on physical properties of the pipeline~\cite{Osiadacz87}. The function $\sign(x)$ returns $+1$ if $x>0$; $-1$ if $x<0$; and $0$ if $x=0$. The absolute value in \eqref{seq:wey2a} signifies that pressure drops along the direction of flow. In particular, the drop in squared pressures is proportional to the squared flow. We will henceforth refer to $\psi_m$ as pressure rather than squared pressure for brevity. Let us collect all $\psi_n$'s in $\bpsi\in\mathbb{R}^N$. 

To enable the desired flow of gas in an NGN while maintaining pressures within acceptable limits, system operators install compressors at selected pipelines. A pipeline hosting a compressor can be modeled by an ideal compressor which increases the gas pressure, followed by a lossy pipeline that incurs a pressure drop per \eqref{eq:wey2}. Apparently, the gas flows on the two edges are identical. Let the subset of edges hosting ideal compressors be $\mcP_a\subset\mcP$. The edges in $\mcP_a$ are also referred to as \emph{active pipelines}. The pressures across an active pipeline or compressor $\ell=(m,n)\in\mcP_a$ are related as
\begin{subequations}\label{eq:comp}
	\begin{align}
	\psi_n&=\alpha_\ell\psi_m\label{seq:compa}\\
	\phi_\ell&\geq0,\label{seq:compb}
	\end{align}
\end{subequations}
where $\alpha_\ell>0$ is the multiplicative compression factor for compressor $\ell$. The unidirectional flow permitted for a compressor is enforced by \eqref{seq:compb}. The remaining edges, that is the edges not hosting ideal compressors, constitute the set $\bmcP_a:=\mcP\setminus\mcP_a$ and abide by \eqref{eq:wey2} instead of \eqref{eq:comp}.

In an NGN, a node $r\in\mcN$ is selected as a reference node. Its pressure is kept fixed. Given $\psi_r$, if the nodal pressures $\bpsi$ are known, the flows $\bphi$ can be readily computed; and vice versa. This fact follows immediately from \eqref{eq:wey2}--\eqref{eq:comp}, and is itemized as the next lemma to be used in subsequent arguments.

\begin{lemma}\label{le:fp}
Given a reference pressure $\psi_r$ for some $r\in\mcN$, a pair $(\bphi,\bpsi)$ satisfying \eqref{eq:wey2}--\eqref{eq:comp} is uniquely characterized by either $\bphi$ or $\bpsi$. 
\end{lemma}

The task of finding $\bphi$ or $\bpsi$ given a combination of nodal injections and pressures constitutes the \emph{gas flow} (GF) problem. Oftentimes, gas supply nodes are tuned to maintain a fixed pressure while injecting variable amounts of gas to meet the prescribed pressure under variable demands~\cite{Zlotnik2017coordinated}, \cite{QLi2003solvingGF}. Let set $\mcN_\psi\subset \mcN$ consist of all nodes with fixed pressures $\psi_n$'s. The reference node $r$ belongs to $\mcN_\psi$ by definition. Its complement set $\mcN_q:=\mcN\setminus \mcN_\psi$ consists of all nodes with fixed injections $q_n$'s. Then, the GF problem can be formally stated now.

\begin{definition}\label{de:GF}
Given pressures $\psi_n$ for $n\in\mcN_\psi$; injections $q_n$ for $n\in\mcN_q$; the ratios $\alpha_\ell$ for all compressors $\ell\in\mcP_a$; and the friction parameters $a_\ell$ for all lossy pipes $\ell\in\bmcP_a$, the GF problem aims at finding the triplet $(\bpsi,\bphi,\bq)$ satisfying the GF equations \eqref{eq:mc2}--\eqref{eq:comp}. 
\end{definition}

The GF task involves $N-1+P$ equations over $N-1+P$ unknowns. It can be posed as the feasibility problem
\begin{align*}\label{eq:G1}
\mathrm{find}~&~\{\bq,\bphi,\boldsymbol{\psi}\}\tag{G1}\\
\mathrm{s.to}~&~ \eqref{eq:mc2}-\eqref{eq:comp}\\
~&~\textrm{given}~\{q_n\}_{n\in\mcN_q}~\text{and}~\{\psi_n\}_{n\in\mcN_\psi}.
\end{align*}
Albeit \eqref{eq:mc2} and \eqref{eq:comp} are linear, the piecewise quadratic Weymouth equation in \eqref{eq:wey2} is non-convex, while the requirement $\{\phi_\ell \geq 0\}_{\ell\in\mcP_a}$ further complicates the task. The GF problem is typically solved using the Newton-Raphson's method, yet its convergence depends on the initialization~\cite{esquivel2012unified}, \cite{QLi2003solvingGF}, \cite{abhi2017GM}. Commercially available software require careful manual tuning by gas network operator personnel, though that could be attributed to more detailed models of NGN components.

A popular rendition of the GF problem considers the reference node as the only fixed-pressure node, and all other nodes as fixed-injection nodes~\cite{misra2015GP}, \cite{abhi2017GM}, \cite{Singh18ACC}. For this rendition, solving the GF problem becomes trivial for a tree network by inverting \eqref{eq:mc2} and using Lemma~\ref{le:fp}. However, for a meshed NGN, solving the GF problem remains non-trivial. Before developing new GF solvers, the next section establishes that the GF problem in \eqref{eq:G1} enjoys a unique solution. 

\section{Uniqueness of the GF Solution}\label{sec:unique}
We commence with the uniqueness of the GF task under the setup of a single fixed-pressure node, proved in~\cite[Th. 1]{Singh18ACC}. 

\begin{theorem}[\cite{Singh18ACC}]\label{th:GFunique1}  
If $\mcN_\psi=\{r\}$ and $\mcN_q=\mcN\setminus \{r\}$, the gas flow problem \eqref{eq:G1} has a unique solution, if feasible.
\end{theorem}

Although the single fixed-pressure setup has been studied widely, setups with multiple fixed-pressure nodes are of critical interest too. This is because gas is typically injected at supplier sites using a controller that maintains constant pressure, rather than constant rate. To address this need, this section builds on Th.~\ref{th:GFunique1} and establishes the uniqueness of the steady-state GF equations for any $(\mcN_\psi,\mcN_q)$ setup. Before doing so, let us briefly review some graph theory preliminaries. 

A directed graph $\mcG=(\mcN, \mcP)$ is connected if there exists a sequence of adjacent edges between any two nodes. All graphs considered in this work are assumed to be connected. A sequence of adjacent edges between nodes $m$ and $n$ constitutes a \emph{path} $\mcP_{mn}\subset\mcP$. The directionality assigned to path $\mcP_{mn}$ is from $m$ to $n$. Note that nodes $m$ and $n$ could be connected by multiple paths. Thus, with slight abuse in notation, path $\mcP_{mn}$ shall represent any arbitrary path between $m$ and $n$, unless additional conditions are provided. For path $\mcP_{mn}$, we can define an \emph{indicator vector} $\bpi^{mn}\in\{0,1\}^P$ with $\ell-$th entry
\begin{equation*}
\pi^{mn}_{\ell}:=
\begin{cases}
0&,~\text{if edge } \ell\notin \mcP_{mn}\\
+1&,~\text{if direction of }\ell\text{ agrees with path direction}\\
-1&,~\text{otherwise.}\\
\end{cases}
\end{equation*}


A \emph{cycle} is a sequence of adjacent edges (without edge or node repetition) that starts and ends at the same node. With a slight abuse of terminology, the statement `cycle $\mcC$ \emph{contains} node $i$' will mean that there exists an edge in $\mcC$ that is incident to node $i$. For any cycle $\mcC$, we can select an arbitrary direction and define its indicator vector $\bn^\mcC$ with $\ell$-th entry
\begin{equation*}
n^\mcC_{\ell}=
\begin{cases}
0&,~\text{if edge } \ell\notin \mcC\\
+1&,~\text{if direction of }\ell\text{ agrees with cycle direction}\\
-1&,~\text{otherwise.}\\
\end{cases}
\end{equation*}
A \emph{tree} is a connected graph with no cycles. 


After the graph theoretic preliminaries, we proceed with the uniqueness of the GF solution for the general GF setup. This proof builds upon the ensuing two lemmas, which are proved in the appendix.
 
\begin{lemma}\label{le:1path}
Consider path $\mcP_{mn}$ along edges $\{\ell_1,\dots,\ell_k\}$ with indicator $\bpi^{mn}$. For fixed pressures $\psi_m$ and $\psi_n$, if flow vectors $\bphi$ and $\bphi'$ with $\bphi\neq \bphi'$, satisfy \eqref{eq:wey2}--\eqref{eq:comp}, they cannot satisfy 
\begin{subequations}\label{eq:le2}
\begin{align}	
\sign(\bphi'-\bphi)\odot\bpi^{mn}&>\bzero~~\textrm{or}~\label{eq:le2a}\\
\sign(\bphi'-\bphi)\odot\bpi^{mn}&<\bzero \label{eq:le2b}
\end{align}
\end{subequations}
where the strict inequalities are understood entrywise. 	
\end{lemma} 

To get some intuition, suppose that $\bpi^{mn}$ takes the value of $+1$ for edges $\{\ell_1,\dots,\ell_k\}$, and $0$ for the remaining edges. According to Lemma~\ref{le:1path}, if two pairs $(\bphi,\bpsi)$ and $(\bphi',\bpsi')$ satisfy \eqref{eq:wey2}--\eqref{eq:comp} with $\psi_m=\psi_m'$ and $\psi_n=\psi_n'$, then the flows along $\mcP_{mn}$ cannot uniformly increase from $\bphi$ to $\bphi'$. In other words, $\phi_\ell'>\phi_\ell$ cannot occur simultaneously for all $\ell\in\mcP_{mn}$. Flows cannot uniformly decrease either ($\phi_\ell'<\phi_\ell$ for all $\ell\in\mcP_{mn}$). This holds merely because the pressure drop across a pipe decreases monotonically with gas flow from and compressors perform a linear scaling [cf.~\eqref{eq:wey2} and \eqref{eq:comp}].

The next lemma describes an interesting effect on how gas flows get redistributed when gas injections change. 

\begin{lemma}\label{le:networkflow}
Consider two pairs $(\bq,\bphi)$ and $(\bq',\bphi')$ satisfying \eqref{eq:mc2}. If $\bq\neq\bq'$, there exists a path $\mcP_{mn}$ between nodes $m$ and $n$ such that
\begin{subequations}\label{eq:netflow}
\begin{align}
&\sign(\bphi'-\bphi)\odot\bpi^{mn}>\bzero\label{seq:netflowa}\\
&q_m'>q_m~~\text{and}~~ q_n'<q_n\label{seq:netflowb}
\end{align}
\end{subequations}
where $\bpi^{mn}$ is the indicator vector for $\mcP_{mn}$.
\end{lemma}

Lemma~\ref{le:networkflow} predicates that if gas injections change, there exists a path: \emph{i)} along which flows increase uniformly; \emph{ii)} the source node of the path has increased injection; and \emph{iii)} the destination node has decreased injection. Lemma~\ref{le:networkflow} has been established in~\cite{Vuffray2015monotonicity} via mathematical induction; see the appendix for an alternative perhaps more intuitive proof.

Using Theorem~\ref{th:GFunique1} and Lemmas~\ref{le:1path}--\ref{le:networkflow}, we next prove the uniqueness of the GF task under the general setup.

\begin{theorem}\label{th:GFunique}  
The gas flow problem \eqref{eq:G1} has a unique solution, if feasible.
\end{theorem}

\begin{proof}
Proving by contradiction, assume $(\bq,\bphi,\bpsi)$ and $(\bq',\bphi',\bpsi')$ are two distinct solutions of \eqref{eq:G1}. Consider the GF setup where $|\mcN_\psi|>1$; the special case of $|\mcN_\psi|=1$ is covered by Theorem~\ref{th:GFunique1}. If $\bq\neq\bq'$, then Lemma~\ref{le:networkflow} implies that there exists a path $\mcP_{mn}$ with indicator vector $\bpi^{mn}$ satisfying \eqref{seq:netflowa}. Moreover, it holds that $q_m'>q_m$ and $q_n'<q_n$ from \eqref{seq:netflowb}. By definition, gas injections $q_i$ are fixed for all nodes $i\in\mcN_q$. Therefore, nodes $m$ and $n$ cannot be fixed-injection nodes. They have to be fixed-pressure nodes belonging to $\mcN_\psi$, implying $\psi_m=\psi_m'$ and $\psi_n=\psi_n$. However, with the pressures at nodes $m$ and $n$ fixed, the inequality \eqref{seq:netflowa} contradicts Lemma~\ref{le:1path}. Hence, the assumption of unequal injections is refuted, implying $\bq=\bq'$.

Given $\bq$ and the reference pressure $\psi_r$, Theorem~\ref{th:GFunique1} asserts that there is unique triplet $(\bq,\bphi,\bpsi)$ satisfying the GF equations. Since $\bq=\bq'$, the triplets $(\bq,\bphi,\bpsi)$ and $(\bq',\bphi',\bpsi')$ have to coincide, which completes the proof.
\end{proof}

The uniqueness claim of Theorem~\ref{th:GFunique} is fairly general, since it applies to any NGN topology and any GF setup with a single or multiple fixed-pressure nodes. Having established uniqueness, the next two sections develop a suite of GF solvers: Section~\ref{sec:efm} builds upon an existing convex solver for GF setups with a single fixed-pressure node and no compressors. We develop an unconstrained convex solver as well as an extension that handles compressors on non-overlapping cycles. Section~\ref{sec:cr} adopts a convex relaxation and puts forth an MI-QCQP to handle more general GF setups. The relaxation is provably exact for NGNs with a single fixed-pressure node and non-overlapping cycles. Nonetheless, numerical tests demonstrate that this MI-QCQP succeeds in finding the unique GF solution in NGNs with multiple fixed-pressure nodes and overlapping cycles as long as compressors are not on overlapping cycles. 

\section{Energy Function Minimization}\label{sec:efm}
This section studies the GF task for the special case of $|\mcN_\psi|=1$. In an NGN without compressors, the GF task is posed as a convex minimization. The approach can be extended to networks having compressors, but not on cycles.

\subsection{Existing Constrained Energy Function-based GF Solver}\label{subsec:efm:existing}
Consider solving the GF task for a single fixed-pressure node (the reference node $r$) and in an NGN without compressors. This task boils down to solving equations \eqref{eq:mc2}--\eqref{eq:wey2}. As shown in~\cite{wolf2000energy}, the gas flows $\bphi$ for this GF setup can be found as the minimizer of the convex minimization
\begin{subequations}\label{eq:Nesterov}
		\begin{align}
		\min_{\bphi}~&~\sum_{\ell\in\mcP} \frac{a_\ell}{3} |\phi_\ell|^3 \label{eq:Nesterov:cost}\\
		\mathrm{s.to}~&~\bA^{\top} \bphi = \bq. \label{eq:Nesterov:con}
		\end{align}
\end{subequations}
This can be readily verified by the first-order optimality conditions of \eqref{eq:Nesterov}. In addition, the pressures $\bpsi$ can be recovered from the optimal Lagrange multipliers $\bxi\in\mathbb{R}^N$ associated with constraint \eqref{eq:Nesterov:con}: If $\bxi$ is shifted by a constant so that its $r$-th entry equals $\psi_r$, the remaining entries of this shifted $\bxi$ equal $\bpsi$. Problem \eqref{eq:Nesterov} can be reformulated as a second-order cone program or tackled via dual decomposition; see~\cite{SinghKekatosWF19}. 

\subsection{Novel Unconstrained Energy Function-based GF Solver}\label{subsec:efm:novel}
Rather than solving \eqref{eq:Nesterov} over $\bphi$, here we show that one can alternatively find the GF solution via an unconstrained convex minimization over $\bpsi$ as
\begin{equation}\label{eq:uGF}
\min_{\bpsi}~\frac{2}{3}\sum_{(m,n)\in\mcP}\frac{|\psi_m-\psi_n|^\frac{3}{2}}{\sqrt{a_{mn}}} - \bq^\top\bpsi.
\end{equation}
The convexity of this objective function follows from composition rules. Since this function is convex and differentiable, its unconstrained minimization is equivalent to nulling its gradient vector. Setting the $n$-th entry of this gradient to zero reveals that the minimizer $\bpsi^*$ of \eqref{eq:uGF} satisfies
\begin{equation}\label{eq:uGFgradient}
\sum_{\ell=(m,n)\in\mcP}\sign(\ba_\ell^\top \bpsi^*)\sqrt{\frac{|\ba_\ell^\top \bpsi^*|}{a_\ell}}=q_n
\end{equation}
where $\ba_\ell^\top$ is the $\ell$-th row of matrix $\bA$. Equation \eqref{eq:uGFgradient} is equivalent to eliminating the flows $\bphi$ from \eqref{eq:mc2} and \eqref{eq:wey2}. As with \eqref{eq:Nesterov}, the ambiguity in pressures could be handled by shifting $\bpsi^*$ by a constant, so that $\psi_r^*$ agrees with the given pressure at the reference node $r$. Once pressures $\bpsi^*$ have been determined, flows can be found using Lemma~\ref{le:fp}.

\begin{remark}\label{re:re1}
In the absence of compressors and when $|\mcN_\psi|=1$, the GF task becomes structurally similar to the water flow problem in water distribution networks without pumps~\cite{SinghKekatosWF19}. Therefore, the (un)-constrained energy function minimization approaches of \eqref{eq:Nesterov}--\eqref{eq:uGF} apply to the gas flow and water flow problems alike. For water networks, the decomposition technique of \cite{SinghKekatosWF19} extends \eqref{eq:Nesterov}--\eqref{eq:uGF} to water network setups with $|\mcN_\psi|=1$ and pumps, but pumps cannot lie on cycles. A similar technique can be used to solve the GF problem with compressors not on cycles and $|\mcN_\psi|=1$. The only modification needed relates to accounting for the multiplicative pressure law in gas compressors [cf.~\eqref{seq:compa}] vis-\`{a}-vis the additive pressure law of water pumps. Additionally, the decomposition algorithm may be extended to accommodate compressors on non-overlapping cycles using the flow-recovery procedure provided later as Algorithm~\ref{algo:1}. Since carrying over this decomposition technique from the water flow to the gas flow context is straight-forward and due to space limitations, it is not presented here.
\end{remark}

To handle GF setups with $|\mcN_\psi|>1$ and/or NGNs with compressors in loops, a convex relaxation of the Weymouth equation is pursued in the next section.

\section{MI-QCQP Relaxation}\label{sec:cr}
The minimization approaches of \eqref{eq:Nesterov}--\eqref{eq:uGF} provide computationally efficient methods to solve the GF problem, but exhibit three limitations: \emph{i)} they cannot handle multiple fixed-pressure nodes ($|\mcN_\psi|>1$); \emph{ii)} cannot handle compressors on cycles; and \emph{iii)} cannot be extended to optimal gas flow formulations (e.g., along the lines of \cite{singh2018optimal}). To overcome these limitations, this section presents an MI-QCQP-based solver that is applicable to any GF setup.

\subsection{Problem Reformulation}\label{subsec:cr:problem}
The non-convexity of \eqref{eq:G1} is due to the Weymouth equation in \eqref{seq:wey2a}. The piecewise quadratic equalities can be relaxed to convex inequality constraints: The pressure drop along a lossy pipe $\ell=(m,n)\in\bmcP_a$ is relaxed to
\begin{itemize}
	\item $\psi_m-\psi_n\geq a_\ell\phi_{\ell}^2$ for $\phi_{\ell}\geq 0$; or
	\item $\psi_n-\psi_m\geq a_\ell\phi_{\ell}^2$ for $\phi_{\ell}\leq 0$.
\end{itemize}  
The two cases can be differentiated using a binary variable $x_\ell$ capturing the direction of flow $\phi_{\ell}$. The relaxed pressure drop equations can be compactly written as
\[(2x_\ell-1)\psi_m+(1-2x_\ell)\psi_n\geq a_\ell\phi_{\ell}^2\]
where $x_\ell=1$ corresponds to $\phi_{\ell}\geq 0$; and $x_\ell=0$ to $\phi_{\ell}\leq0$. Despite the relaxation, the bilinear terms $x_\ell\psi_m$ make the aforementioned constraint non-convex. 

The McCormick linearization, popular for approximating multilinear terms by their linear convex envelopes, can be used to handle these bilinear terms~\cite{McCormick1976}. For the special case of bilinear terms involving at least one binary term, the McCormick linearization becomes exact. In fact, it is related to the so termed big-$M$ trick, but instead of using a single arbitrarily large value for $M$, it selects different values for $M$ that are specialized per product of variables, which could potentially reduce the running time of mixed-integer programming solvers. Let us briefly review the linearization. Consider the constraint $z_{\ell n}=x_\ell\psi_n$, for which $x_\ell\in\{0,1\}$ and $\psi_n\in[\underline{\psi}_m,\overline{\psi}_n]$. This constraint can be equivalently expressed via four linear inequalities
\begin{subequations}\label{eq:MC}
	\begin{align}
	x_{\ell} \underline{\psi}_n&\leq z_{\ell n}\leq x_{\ell}\overline{\psi}_n\label{seq:MCa}\\
	\psi_n+(x_{\ell}-1)\overline{\psi}_n&\leq z_{\ell n}\leq\psi_n+(x_\ell-1)\underline{\psi}_n.\label{seq:MCb}
	\end{align}
\end{subequations}
To verify the exactness, observe that when $x_\ell=1$, constraint \eqref{seq:MCb} yields $z_{\ell n}=\psi_n$ and \eqref{seq:MCa} holds trivially. When $x_\ell=0$, constraint \eqref{seq:MCa} enforces $z_{\ell n}=0$ and \eqref{seq:MCb} holds trivially. Hence, the constraints in \eqref{eq:MC} ensure that $z_{\ell n}=x_\ell\psi_n$.

To arrive at an MI-QCQP relaxation of \eqref{eq:G1}, for all lossy pipes $\ell\in\bmcP_a$, the pressure drop constraint of \eqref{seq:wey2a} is replaced by \eqref{eq:MC} and
\begin{subequations}\label{eq:weyMC}
	\begin{align}
	&2z_{\ell m}-2z_{\ell n}+\psi_n-\psi_m\geq a_\ell\phi_{\ell}^2,\label{seq:weyMCa}\\
	&-\overline{\phi}_\ell(1-x_\ell)\leq \phi_\ell\leq \overline{\phi}_\ell x_\ell,\label{seq:weyMCb}
	\end{align}
\end{subequations}
where $\overline{\phi}_{\ell}$ is an upper bound on $|\phi_{\ell}|$. Constraint \eqref{seq:weyMCa} represents the relaxed Weymouth equation, and constraint \eqref{seq:weyMCb} defines $x_\ell=\sign(\phi_\ell)$. Similar relaxations have been previously used in~\cite{backhaus2016convex}, \cite{bent2016hicss}, \cite{Singh18ACC}; see Section~\ref{sec:Comparison} for a detailed comparison.

When solving the GF problem with the Weymouth equations relaxed, the obtained solution is useful only if the relaxation is \emph{exact}, that is when \eqref{seq:weyMCa} holds with equality for all $\ell$. To render the relaxation provably exact, we convert the feasibility problem \eqref{eq:G1} to the MI-QCQP minimization 
\begin{align*}\label{eq:G2}
\min~&~r(\boldsymbol{\psi})\tag{G2}\\
\mathrm{over}~& ~ \bq, \bphi,\boldsymbol{\psi},\bx\\
\mathrm{s.to}~&~ \eqref{eq:mc2},\eqref{eq:comp},\eqref{eq:MC}, \eqref{eq:weyMC}.
\end{align*}
The optimization variable $\bx$ stacks $\{x_\ell\}_{\ell\in\bmcP_{a}}$, and the objective function is judiciously selected as
\begin{equation*}
r(\boldsymbol{\psi}):=\sum_{\substack{(m,n)\in \bmcP_{a}\\ (m,n)\notin \mcS_\mcC^a}}|\psi_m-\psi_n|
\end{equation*}
where $\mcS_\mcC^a$ is the set of cycles with compressors. These cycles will be also termed as \emph{active cycles}. The cost $r(\boldsymbol{\psi})$ sums up the absolute pressure differences across all lossy pipes not in active cycles. Despite the non-convexity of \eqref{eq:G2} due to the binary variables, this minimization can be handled for moderately sized networks thanks to the advancements in mixed-integer second-order cone solvers. The computational performance of \eqref{eq:G2} is further corroborated by our tests. The next section provides network conditions under which the exactness of \eqref{eq:G2} can be guaranteed analytically. The tests in Section~\ref{sec:tests} demonstrate numerically that solving \eqref{eq:G2} renders the relaxation exact for a much broader class of networks.

For solving tasks such as \eqref{eq:G1}, NR-based or fixed-point iteration solvers are often preferred as opposed to optimization-based solvers due to computational superiority. However, in addition to guaranteeing convergence irrespective of initialization, problem \eqref{eq:G2} can also be used as follows:
\begin{itemize}
	\item \emph{Infeasibility:} As a relaxation of \eqref{eq:G1}, \eqref{eq:G2} can be used to screen infeasible GF instances; see Section~\ref{sec:tests} for tests. Such screening is of practical use as suggested in \cite{Ding2019Gasletter}.
	\item \emph{Initialization:} Problem \eqref{eq:G2} could be terminated before reaching optimality to yield initializations for NR solvers, hence combining the benefits of both approaches.
	\item \emph{Optimal gas flow:} The cost of \eqref{eq:G2} could be useful as a penalty term that can be added to optimization problems~\cite{SinghKekatosWF19}, \cite{singh2018optimal}. However, guaranteeing exact relaxation for such problems would need further analysis.
\end{itemize} 

\subsection{Exactness of the Relaxation}\label{subsec:cr:exact}
The relaxation in \eqref{eq:G2} will be analytically shown to be exact under the following network conditions.

\begin{condition}\label{con:C1}
The GF setup has a single fixed-pressure node, that is $|\mcN_\psi|=1$.
\end{condition}

\begin{condition}\label{con:C2}
Each edge of the NGN belongs to at most one cycle.
\end{condition}

\begin{condition}\label{con:C3}
The NGN does not exhibit circulation of gas, that is $\bn^c\odot\bphi\not>\bzero$ and $\bn^c\odot\bphi\not<\bzero$ for every cycle $\mcC$.
\end{condition}

Under Condition~\ref{con:C1}, the nodal injections are fixed \emph{a priori} and the GF task aims at finding the associated $(\bpsi,\bphi)$. Albeit Definition~\ref{de:GF} considered the GF task with multiple fixed-pressure nodes, the setup of a single fixed-pressure node is commonly met; see~\cite{misra2015GP}, \cite{abhi2017GM}, \cite{Singh18ACC}, \cite{abmann2019robust}. Regarding Condition~\ref{con:C2}, although it may seem restrictive at the outset, it is satisfied by several practical gas networks~\cite{Benner2018gasbenchmark}. For Condition~\ref{con:C3}, a circulation occurs when gas flows around a cycle along the same direction. It is easy to verify that gas cannot circulate in a cycle without compressors, since the incurred pressure drops along the cycle will all be in the same direction and thus cannot sum up to zero. In cycles with compressors, gas circulation can occur though it would cause an undesirable loss of energy. However, the tests of Section~\ref{sec:tests} demonstrate that the relaxation in \eqref{eq:G2} is exact even in setups where the sufficient Conditions~\ref{con:C1}--\ref{con:C3} are all violated.

The next exactness claim applies to the GF setup with known injections. From Lemma~\ref{le:fp}, we know that solving the GF task is equivalent to finding the correct flows $\bphi$. The next result provides conditions under which \eqref{eq:G2} yields flows $\bphi$ with partially correct entries. An algorithm to retrieve the entire $\bphi$ and thus eventually solve \eqref{eq:G1} is presented afterwards.

\begin{theorem}\label{th:gf1exact}
Let $\bphi$ be the unique flow vector solving \eqref{eq:G1}, and $\bphi'$ the flow vector minimizing \eqref{eq:G2}. Under Cond.~\ref{con:C1}--\ref{con:C3}, it holds $\phi_\ell'=\phi_\ell$ for all edges $\ell$ not belonging to active cycles.
 \end{theorem}
 
Theorem~\ref{th:gf1exact} establishes that the only possible mismatches between $\bphi'$ and $\bphi$ occur only at the edges lying on cycles with compressors. Then, if there are no cycles with compressors, the GF problem is solved correctly; see also \cite[Th. 2]{Singh18ACC}.
 
\begin{corollary}\label{co:gfexact}
Under Cond.~\ref{con:C1}--\ref{con:C2}, for an NGN without compressors in cycles, the minimizer of \eqref{eq:G2} solves \eqref{eq:G1} as well.
\end{corollary}
 
Corollary~\ref{co:gfexact} identifies a setup where \eqref{eq:G2} is equivalent to solving \eqref{eq:G1}. Nonetheless, if there are no compressors in cycles, one would prefer tackling \eqref{eq:G1} using the solvers of Section~\ref{sec:efm}. This is because running the decomposition technique discussed in Remark~\ref{re:re1} and solving \eqref{eq:Nesterov}, are simpler than solving the MI-QCQP of \eqref{eq:G2}.
 
%

\subsection{Recovering the GF Solution}\label{subsec:cr:recovery}
Returning to the general setup, we next provide a procedure to retrieve the solution $\bphi$ of \eqref{eq:G1} given a minimizer $\bphi'$ of \eqref{eq:G2}. From Theorem~\ref{th:gf1exact}, vector $\bphi'$ needs to be corrected only at the entries corresponding to edges in active cycles. To this end, we first put forth an algorithm to correct the flows within a single active cycle, and then delineate the steps to systematically correct the flows for all active cycles of the network. 

Consider an active cycle $\mcC$ with $N_\mcC$ nodes. Let $\psi_0$ be a known pressure on node $0\in\mcC$, and $\bphi_\mcC'$ be the $N_\mcC$-length subvector of $\bphi'$ collecting the flows on $\mcC$. Similarly, let $\bn_\mcC$ be the $N_\mcC$-length subvector of the indicator vector for cycle $\mcC$. The next lemma explains how $\bphi_\mcC$ can be recovered from $\bphi_\mcC'$.

\begin{lemma}\label{le:correction}
Given a known pressure $\psi_0$, and flows $\bphi_\mcC'$ on active cycle $\mcC$ obtained from \eqref{eq:G2}, Algorithm~\ref{algo:1} determines the corrected gas flows $\bphi_\mcC$ such that the relaxed Weymouth equations in \eqref{eq:weyMC} are satisfied with equality.
\end{lemma}

\begin{proof}
Because $\bphi$ and $\bphi'$ both satisfy \eqref{eq:mc2}, it follows that $(\bphi-\bphi')\in\nullspace(\bA^\top)$. Since there are no overlapping cycles, we have that $\bphi_\mcC=\bphi_\mcC'+\lambda_\mcC\bn_\mcC$ for some $\lambda_\mcC\in\mathbb{R}$. To recover $\bphi_\mcC$, we next provide a method for finding $\lambda_\mcC$. 

Suppose one is given a $\lambda\in\mathbb{R}$. Given pressure $\psi_0$ and the candidate flow vector $\bphi_\mcC'+\lambda\bn_\mcC$, one can calculate the pressures along $\mcC$ sequentially using \eqref{seq:wey2a} and \eqref{seq:compa}. Upon completing the cycle, the pressure at node $0\in\mcC$ will be evaluated to the value of $\hat{\psi}_0(\lambda)$. The value $\hat{\psi}_0(\lambda)$ may not be equal to $\psi_0$. Note that for $\lambda>\lambda_\mcC$, it holds that
\[\sign(\bphi_\mcC'-\bphi_\mcC +\lambda\bn_\mcC)\odot\bn_\mcC=\sign\left((\lambda-\lambda_\mcC\right)\bn_\mcC)\odot\bn_\mcC>\bzero.\]
Using the above along with the argument used in the proof of Lemma~\ref{le:1path}, it can be shown that $\hat{\psi}_{0}(\lambda)<\psi_{0}$. In a similar fashion, if $\lambda<\lambda_\mcC$, then $\hat{\psi}_{0}(\lambda)>\psi_{0}$. Therefore, the function $\hat{\psi}_{0}(\lambda)-\psi_{0}$ is monotonic in $\lambda$, and $\hat{\psi}_{0}(\lambda)=\psi_{0}$ if and only if $\lambda=\lambda_\mcC$. Thanks to this monotonicity, one can find $\lambda_\mcC$ iteratively using bisection, tabulated as Algorithm~\ref{algo:1}.
\end{proof}

Lemma~\ref{le:correction} shows that $\bphi_\mcC$ can be recovered from $\psi_0$ and $\bphi_\mcC'$ using a bisection technique on $\lambda$. The limits for the search space $[\underline{\lambda}, \overline{\lambda}]$ of $\lambda$ can be found using engineering constraints on gas flows. In fact, these limits can be tightened since the entries of $\bphi_\mcC'$ and $\bphi_\mcC(\lambda)=\bphi_\mcC'+\lambda \bn_\mcC$ corresponding to any compressor in $\mcC$ must have the same sign due to \eqref{seq:compb}. 

We next provide the steps to find the correct GF solution using the flow $\bphi'$ obtained from \eqref{eq:G2}:

\emph{T1)} Select a spanning tree $\mcT$ of the NGN graph $\mcG$ rooted at the reference node $r$.

\emph{T2)} Starting from node $r$, traverse $\mcT$ via a depth-first search. 

\emph{T3)} If a node $n$ does not belong to an active cycle of $\mcG$, calculate its pressure as follows: If the edge connecting node $n$ to its parent node in $\mcT$ is a lossy pipe, use \eqref{seq:wey2a}; else, if this edge is a compressor, use \eqref{seq:compa}.

\emph{T3)} If a node $n$ belongs to an active cycle $\mcC$, check if the flows in cycle $\mcC$ have been corrected. If the flows are already corrected or if $i$ is the first node in $\mcC$ that is encountered, compute the nodal pressure as in step \emph{T3)}. Else, pass the pressure at the parent node of $i$ (which is also in $\mcC$) along with the non-corrected flow subvector $\bphi_\mcC'$ to Algorithm~\ref{algo:1} and obtain the corrected flows on $\mcC$.

\emph{T4)} Continue until all nodes in $\mcT$ have been traversed. 

\begin{algorithm}[t]
	\caption{Recover flows on active cycles upon solving \eqref{eq:G2}}\label{algo:1}
	\SetAlgoLined
	\SetKwInOut{Input}{Input}
	\SetKwInOut{Output}{Output}
	\SetKwInOut{Return}{Return}
	\SetKwInOut{Initialize}{Initialize}
	\Input{$\psi_0,\bphi_\mcC',\bn_\mcC,\underline{\lambda}, \overline{\lambda}$; tolerance $\epsilon$; pipe and compressor parameters along $\mcC$}
	\Output{flow vector $\bphi_\mcC$ and pressure vector $\bpsi_\mcC$ along $\mcC$}
	\Initialize{Set $\lambda\leftarrow\frac{\underline{\lambda}+\overline{\lambda}}{2}$ and $\psi_0(\lambda)\leftarrow\infty$}
	\While{$|\psi_0(\lambda)-\psi_0|\geq\epsilon$}{
		Try flow vector $\bphi_\mcC'+\lambda\bn_\mcC$. Starting from node $0$, compute pressures $\psi_n(\lambda)$ along all nodes in $n\in\mcC$ using \eqref{seq:wey2a} and \eqref{seq:compa} until you return to node $0$. 
		
		\eIf{$\psi_0(\lambda)>\psi_0$}{
			Set $\underline{\lambda} \leftarrow \lambda,~\lambda\leftarrow\frac{\underline{\lambda}+\overline{\lambda}}{2}$
		}{
			Set $\overline{\lambda} \leftarrow \lambda,~\lambda\leftarrow\frac{\underline{\lambda}+\overline{\lambda}}{2}$
		}
	}
\Return{flows $\bphi_\mcC=\bphi_\mcC'+\lambda\bn_\mcC$ and pressures $\bpsi_\mcC(\lambda)$}
\end{algorithm}


\section{Comparison to Prior Work}\label{sec:Comparison}
This work puts forth three novel components: \emph{c1)} proving the uniqueness of the GF problem solution under steady-state conditions; \emph{c2)} proposing GF solvers based on energy function minimization; and \emph{c3)} devising a provably exact MI-QCQP relaxation. These components are next contrasted to existing related works:

\emph{c1)~Uniqueness:} For an NGN with no compressors, the GF solution may be found as a minimizer of~\eqref{eq:Nesterov}; see~\cite{wolf2000energy},~\cite{Ding2019Gasletter}. A linearization technique has also been put forth to accelerate solving \eqref{eq:Nesterov}~\cite{Ding2019Gasletter}. References \cite{Vuffray2015monotonicity} and \cite{misra2015maximum} broaden the uniqueness claim for NGNs with \emph{additive} compressors of constant gain. These works formulate strictly convex problems that yield a GF solution; hence proving uniqueness by convexity. However, gas compressors are oftentimes \emph{multiplicative}, so that the previous uniqueness claims do not carry over. In our work~\cite{Singh18ACC}, uniqueness was proved for multiplicative compressors under any network topology, but with a single fixed-pressure node. Theorem~\ref{th:GFunique} generalizes all past claims for NGNs with multiplicative compressors, any topology, and an arbitrary number of fixed-pressure nodes. 


\emph{c2)~Energy Function Minimization:} Problem~\eqref{eq:Nesterov} dates back to~\cite{wolf2000energy}, and has since been used for solving the GF task; verifying the feasibility and uniqueness of a GF instance~\cite{Ding2019Gasletter}; and initializing optimization problems. However, its applicability was limited to NGNs without compressors. As explained in Remark~\ref{re:re1}, this work suggests using \eqref{eq:Nesterov}  to handle NGNs with compressors on non-overlapping cycles. We also present the unconstrained energy function formulation of~\eqref{eq:uGF}.

\emph{c3)~MI-QCQP relaxation of GF:} The key difficulty in solving (optimal) GF problems stems from the non-linear Weymouth equation. A disjunctive convex relaxation of this equation was found to be efficient in~\cite{backhaus2016convex},~\cite{bent2016hicss}. Numerous studies have thereon employed similar convex relaxations; see~\cite{bent2018market},~\cite{Schwele2019coordination},~\cite{chen2018gaselectric}. Unfortunately, it is hard to guarantee the exactness of these relaxations. An effective heuristic is to fix the binary variables involved to the values obtained by the convex relaxation and handle the resultant non-convex nonlinear program through a general solver~\cite{backhaus2016convex},~\cite{bent2018market}. A gas-electric flow problem was solved in~\cite{chen2018gaselectric}, wherein a cost function was proposed that was numerically found useful towards attaining exact relaxation. Unlike previous works, the MI-QCQP formulation of Section~\ref{sec:cr} provides theoretical guarantees for exact relaxation, while expanding the claims of~\cite{Singh18ACC}. The GF solver developed in \cite{Singh18ACC} was applicable to NGN's with compressors not on cycles. However, in this work, the cost function of \eqref{eq:G2} is meticulously designed to ensure that correct flows are obtained outside active cycles. Additionally, Algorithm~\ref{algo:1} is developed to enable flow correction on active cycles efficiently.  Although the GF problem is intrinsically simpler than the optimal gas (and possible electric) flow problem considered in prior works, this work lays a foundation towards analytical guarantees for exact relaxation. It has been recently shown that exact relaxation of network flow optimization problems may be guaranteed using a convex penalty~\cite{Garcia2019networkflow}. It is worth mentioning that a related MI-QCQP formulation of the water flow problem in \cite{SinghKekatosWF19}, can also provably yield an \emph{exact} convex relaxation for the optimal water flow task~\cite{singh2018optimal}.

\begin{figure}[t]
	\centering
	\includegraphics[scale=0.35]{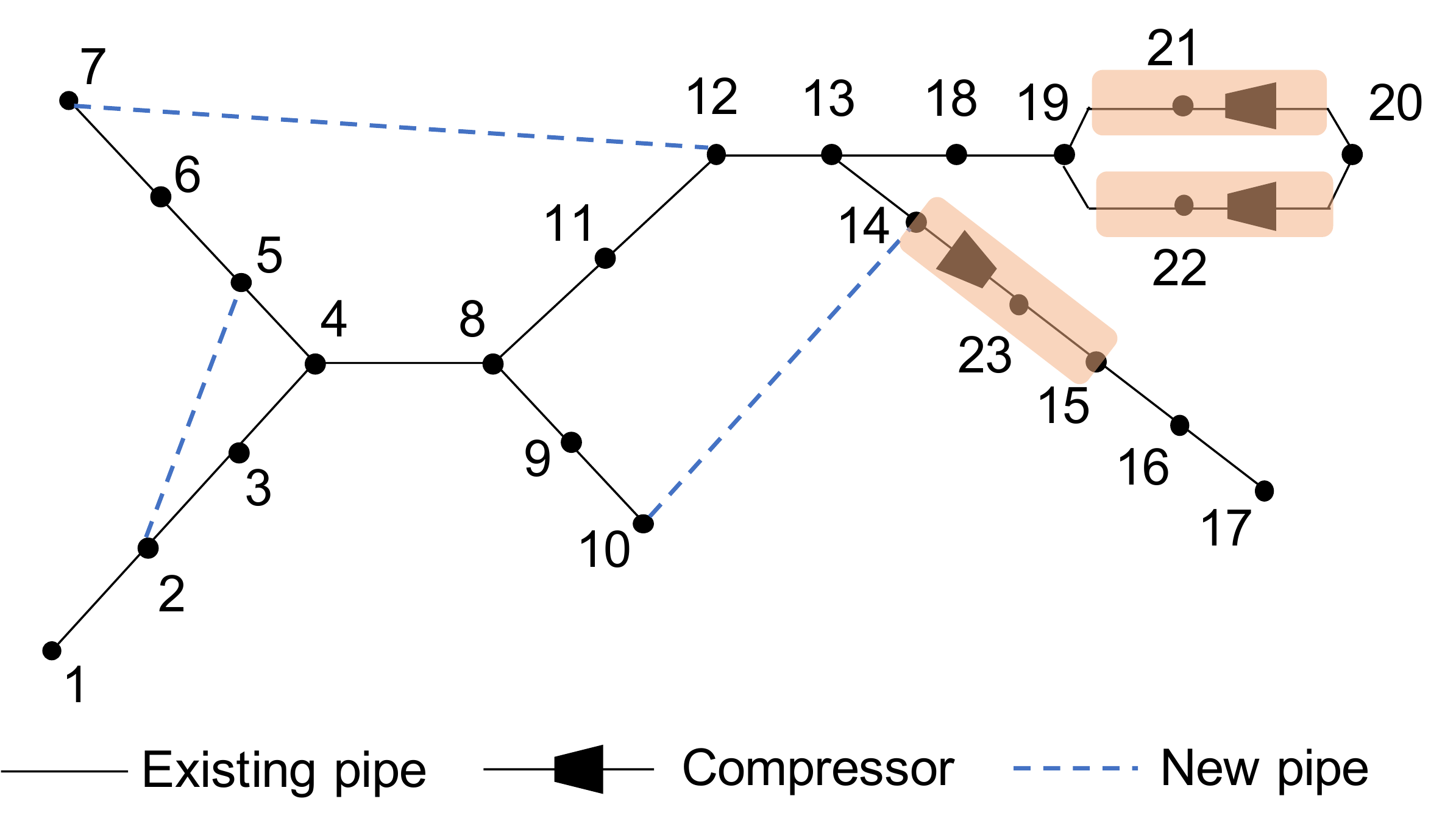}
	\includegraphics[scale=0.35]{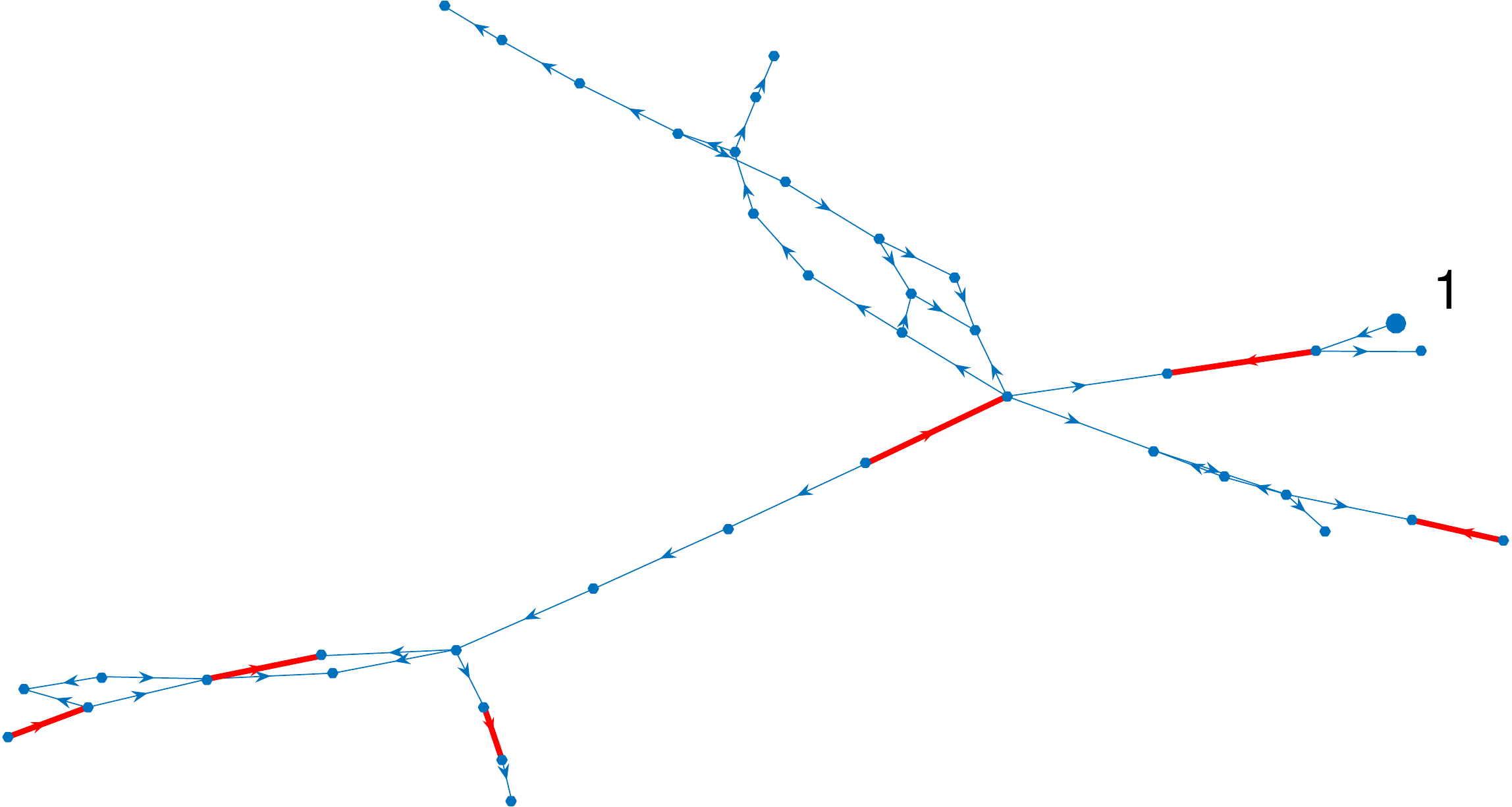}
	\caption{\emph{Top:}~Modified Belgian natural gas network. \emph{Bottom}: GasLib-40 network with red edges representing compressors.}
	\label{fig:belgian}
\end{figure}

\section{Numerical Tests}\label{sec:tests} 
The proposed GF solver based on the relaxed MI-QCQP \eqref{eq:G2} and Algorithm~\ref{algo:1} was tested on the modified Belgian benchmark NGN and the GasLib-40 NGN of Fig.~\ref{fig:belgian}. Starting with the Belgian NGN, the pipe coefficients and compressor ratios were derived based on the nodal pressures and edge flows reported in \cite{wolf2000energy}. The network contains three compressors, which are modeled as ideal compressors  followed by lossy pipes. Problem \eqref{eq:G2} was solved using the MATLAB-based optimization toolbox YALMIP using CPLEX as the MI-QCQP solver~\cite{yalmip},~\cite{cplex}. All tests were conducted on a 2.7 GHz Intel Core i5 computer with 8 GB RAM.

As a model validation step, we first tested the \eqref{eq:G2} solver on the original Belgian network, which is a tree, except for one cycle formed by parallel compressors, see~Fig.~\ref{fig:belgian}. The pressure at node~$1$ was treated as reference. The flow values obtained from \eqref{eq:G2} agreed with those of~\cite{wolf2000energy} for all edges except for the edges along the active cycle. Similarly, the pressures agreed for all nodes other than node~$20$. Therefore, the pressure at node~$19$ and the flows on edges $(19,21)$, $(19,22)$, $(20,21)$, $(20,22)$ were passed to Algorithm~\ref{algo:1} for correction. The final result was found to coincide with \cite{wolf2000energy}.

The Belgian network was subsequently augmented by additional pipelines; see Fig.~\ref{fig:belgian}. The resulting modified network has overlapping cycles, thus violating Condition~\ref{con:C2} required in Theorem~\ref{th:gf1exact}. To get reasonable friction coefficients, for every added line $(m,n)$, the coefficient $a_{mn}$ was set equal to the sum of $a_\ell's$ along the $m-n$ path, yielding $a_{2,5}=0.1936$, $a_{10,14}=0.0439$, $a_{7,12}=0.0419$. We kept the reference pressure at node $1$ and the compression ratios constant as in \cite{wolf2000energy}, and drew $1,500$ random gas injections $\bq$. To construct these samples, we perturbed the benchmark injections $\bq_0$ that lie in the range $[-15.61,22.01]$ by a standardized normal deviation. The injection at node $20$ was set to the negative sum of the remaining injections to get $\bone^\top\bq=0$ for all samples.

Using the modified meshed Belgian NGN of Fig.~\ref{fig:belgian} and the random gas injections, we tested the exactness of \eqref{eq:G2} and the performance of Algorithm~\ref{algo:1}. Not all of the random injections were feasible for the GF problem -- some violated \eqref{seq:compb} or \eqref{seq:wey2b}. Problem \eqref{eq:G2} was infeasible for $876$ out of the $1,500$ random instances. Since \eqref{eq:G2} is a relaxation of \eqref{eq:G1}, these instances are apparently infeasible for \eqref{eq:G1} too. The performance of \eqref{eq:G2} and Algorithm~\ref{algo:1} was tested on the remaining $624$ gas injection instances. To evaluate the success of \eqref{eq:G2} in solving \eqref{eq:G1}, we calculated the \emph{inexactness gap} $G$ defined as
\[G:=\max_{(m,n)\in\bmcP_a}\frac{|\psi_m-\psi_n|-a_{mn}\phi_{mn}^2}{a_{mn}\phi_{mn}^2}\geq 0\]
for the pressures and flows obtained by \eqref{eq:G2} and Algorithm~\ref{algo:1}.

\begin{figure}[t]
	\centering
	\includegraphics[scale=0.22]{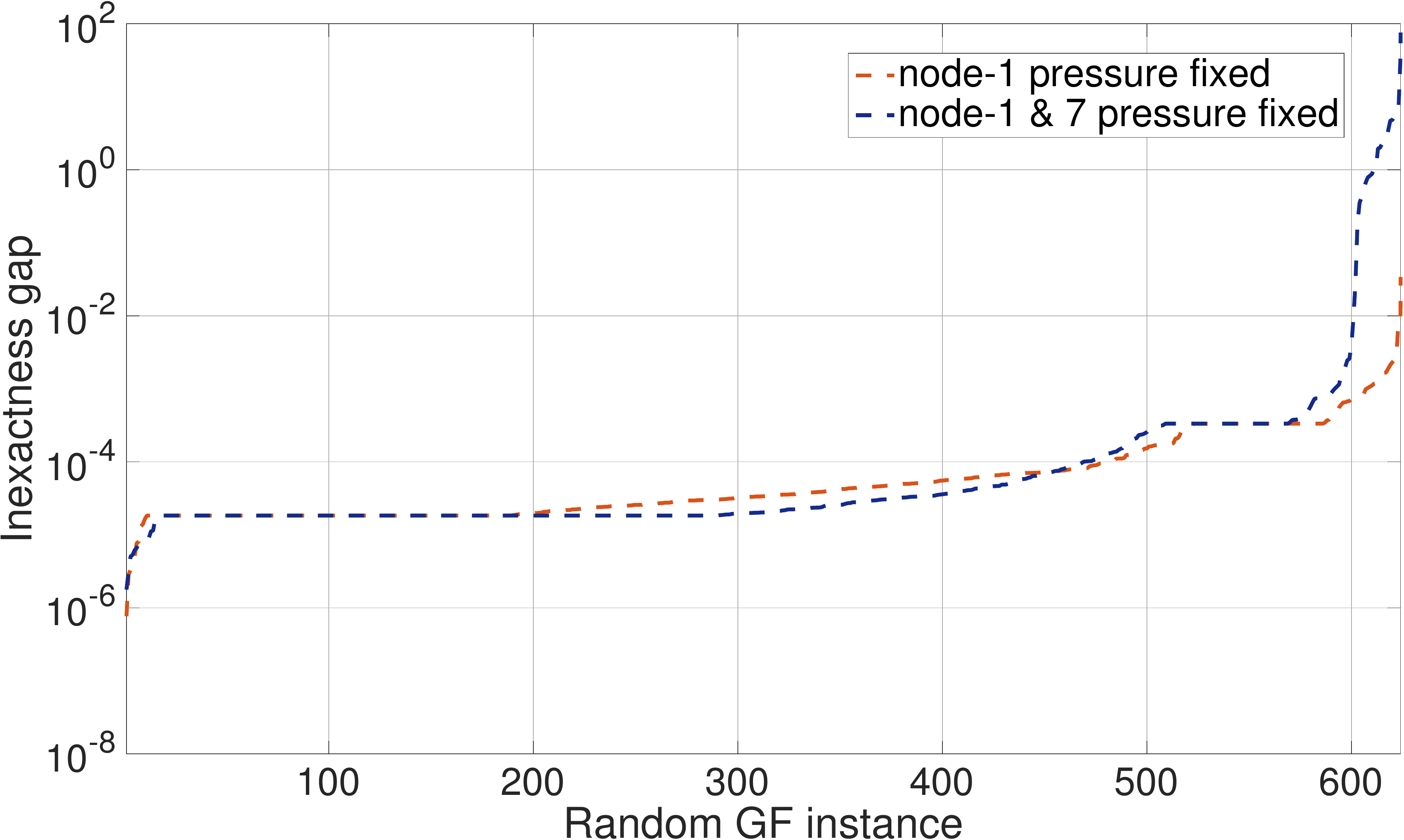}
	\caption{Inexactness gap attained by \eqref{eq:G2} followed by Algorithm~\ref{algo:1} over random feasible instances of the GF problem.}
	\label{fig:feas}
\end{figure}

\begin{figure}[t]
	\centering
	\includegraphics[scale=0.22]{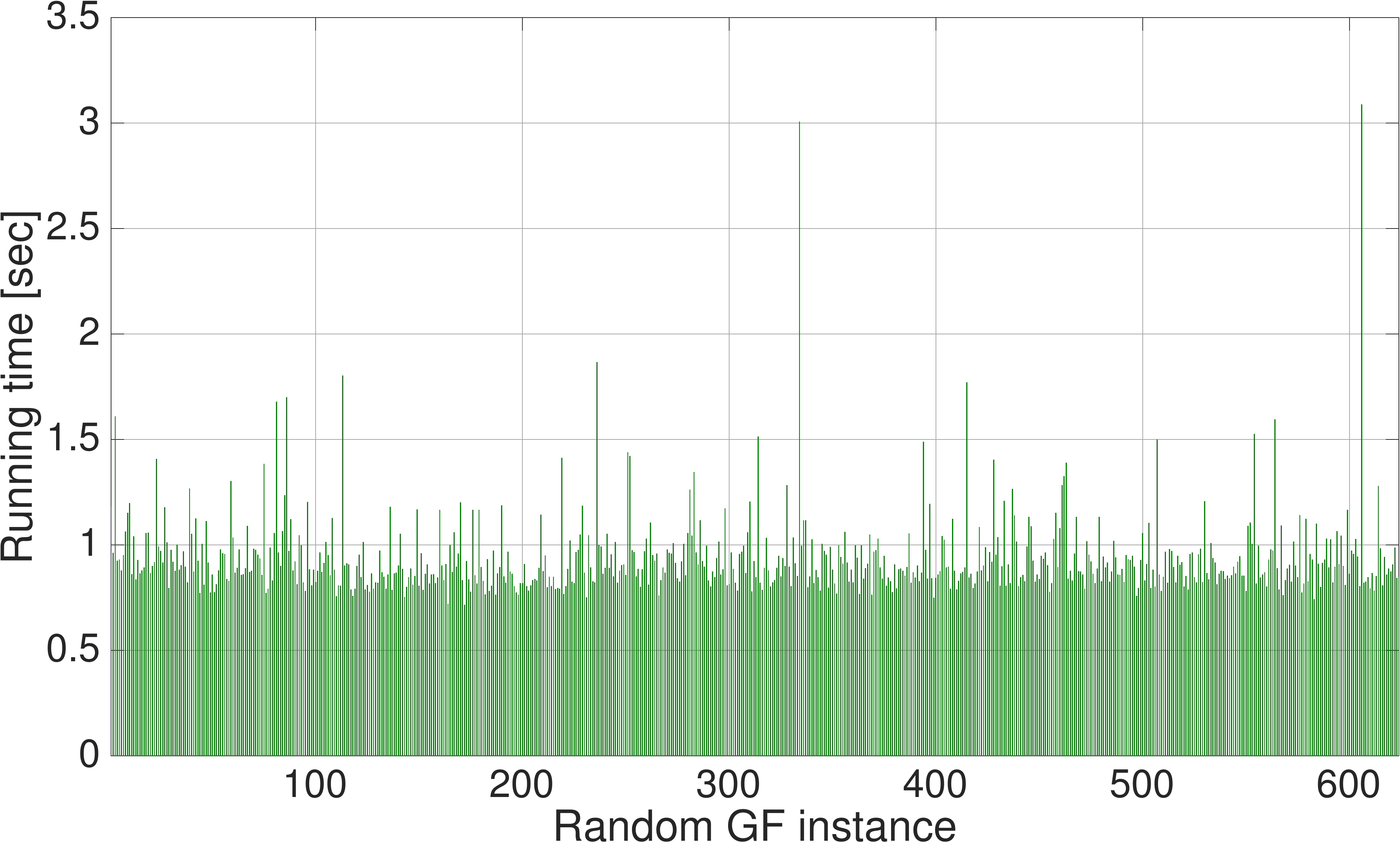}
	\caption{Running time for \eqref{eq:G2} and Alg.~\ref{algo:1} over random feasible GF instances.}
	\label{fig:time}
\end{figure}

The ranked inexactness gap for the feasible GF instances is shown by the first curve in Fig.~\ref{fig:feas}. The gap was less than $10^{-3}$ for more than $97\%$ of the feasible instances, while the maximum gap over all instances was $0.009$. This corroborates that the proposed solver performs well even when Condition~\ref{con:C2} is not met. Fig.~\ref{fig:time} shows the running time for solving \eqref{eq:G2} and Algorithm~\ref{algo:1} over the $624$ feasible instances. The average (median) running time was 0.96 sec (0.89 sec). 

Considering Condition~\ref{con:C1}, we used the fixed pressure at node $1$ and the pressures obtained at node $7$ for the feasible GF instances, we solved \eqref{eq:G2} again. Although the hypothesis of Th.~\ref{th:gf1exact} does not hold anymore, the inexactness gap was found to be less than $10^{-3}$ for more than $94\%$ of the instances; see the second curve in Fig.~\ref{fig:feas}. Thus, the tests reveal that the novel solver successfully finds the GF solution even when the sufficient Conditions~\ref{con:C1}--\ref{con:C2} are violated. However, Condition~\ref{con:C3} prohibiting gas circulations could not be violated for the Belgian NGN because the only active cycle in this NGN has parallel compressors, hence avoiding circulations from~\eqref{seq:compb}. We next deal with GF instances on the GasLib-40 network, wherein a circulation could potentially occur.

GasLib-40 roughly represents a part of the German gas transport network~\cite{schmidt2017gaslib}. The network exhibits 40 nodes, 39 pipes, and 6 compressors; see Fig.~\ref{fig:belgian}~(\emph{Bottom}). The pipe dimensions, roughness coefficients, and a nominal demand vector $\bq_0$ were derived from~\cite{schmidt2017gaslib}. The goals for conducting additional tests on GasLib-40 include: \emph{i)} Evaluating our solvers on a realistic setup; \emph{ii)} Testing our MI-QCQP when Condition~\ref{con:C3} is violated; and \emph{iii)} Benchmarking the performance of our solvers against NR-based solver. We next briefly introduce the NR-based solver used for benchmarking. Given an injection $\bq$, compressor ratios $\alpha_\ell$'s, and reference pressure $\psi_{1}$, stack the unknowns as ${\by=[\phi_1,\dots,\phi_L,\psi_{2},\dots,\psi_N]^\top}$. Define the equality constraints \eqref{eq:mc2},~\eqref{seq:wey2a}, and \eqref{seq:compa} collectively as $\bg(\by)=0$. Given an initial estimate $\by_0$, the NR-based solver would iterate as \[\by_{t+1}=\by_{t}-\mu[\bJ(\by_{t})]^{-1}\bg(\by_{t})\]
where $t$ is the iteration count; matrix $\bJ(\by_{t})$ is the Jacobian of $\bg(\by)$ evaluated at $\by_{t}$; and $\mu$ a step size. A solution $\by^\star$ obtained on convergence of NR updates would be deemed feasible if the inequalities \eqref{seq:wey2b} and \eqref{seq:compb} are satisfied. Since, the NR updates target at attaining $\bg(\by)=\bzero$, the performance evaluation criteria for our results would be $\|\bg(\by)\|_2$ in lieu of the inexactness gap $G$.

In the first set of tests on GasLib-40, we generated $500$ gas injection instances $\bq$ by scaling the entries of $\bq_0$ independently, by random factors chosen uniformly on $[0.75,1.25]$. The pressure at node $1$ was set to $50$~bar and its injection was set to the negative sum of other nodes for all instances. Next, the compression ratios for the 6 compressors were drawn uniformly within $[1,2]$. All $500$ instances were solved using three approaches: \emph{a1)} the MI-QCQP and Algorithm~\ref{algo:1}; \emph{a2)} NR with flows initialized at $(\bA^\top)^\dagger\bq$, and all pressures initialized at $\psi_{1}$; and \emph{a3)} NR with flows and pressures initialized at the solution of MI-QCQP and Algorithm~\ref{algo:1}. The stopping criteria for NR was set to $\|\bg(\by)\|_2<10^{-3}$, subject to a maximum iteration count of $50$. The step size for both initialization scenarios was kept as $\mu=1$. The MI-QCQP deemed $5$ out of the $500$ instances as infeasible and the performance criteria $\|\bg(\by)\|_2$ was found to lie in $[0.005,0.183]$ with the median at $0.009$. To compare to the index of inexactness gap, the range for $G$ for the $495$ feasible cases was $[8\cdot10^{-5},~6\cdot10^{-2}]$. Thus, the MI-QCQP alongside Algorithm~\ref{algo:1} was successful in finding the GF solution for all $495$ instances. Interestingly, $474$ of the $495$ feasible GF instances exhibit circulations, and hence violate Condition~\ref{con:C3}. Thus, the numerical results empirically demonstrate that the developed MI-QCQP alongside Algorithm~\ref{algo:1} successfully solves the GF problem even when the conditions of Theorem~\ref{th:gf1exact} are violated. The NR solver, if initialized at the solution of MI-QCQP improves the solution accuracy, resulting in $\|\bg(\by^\star)\|_2$ within $1.2\cdot10^{-4}-0.13$. For the $5$ instances deemed infeasible by MI-QCQP, the NR solver was initialized at all zero flows and pressures; all $5$ instances failed to converge. Surprisingly, when the NR solver was initialized with $(\bA^\top)^\dagger\bq$ as flows and $\psi_{0}'s$ as pressures, all $500$ instances failed to converge. The non-convergence of the NR solver is however alleviated when $\mu$ was reduced as discussed next.

\begin{figure}[t]
	\centering
	\includegraphics[scale=0.33, angle=-90]{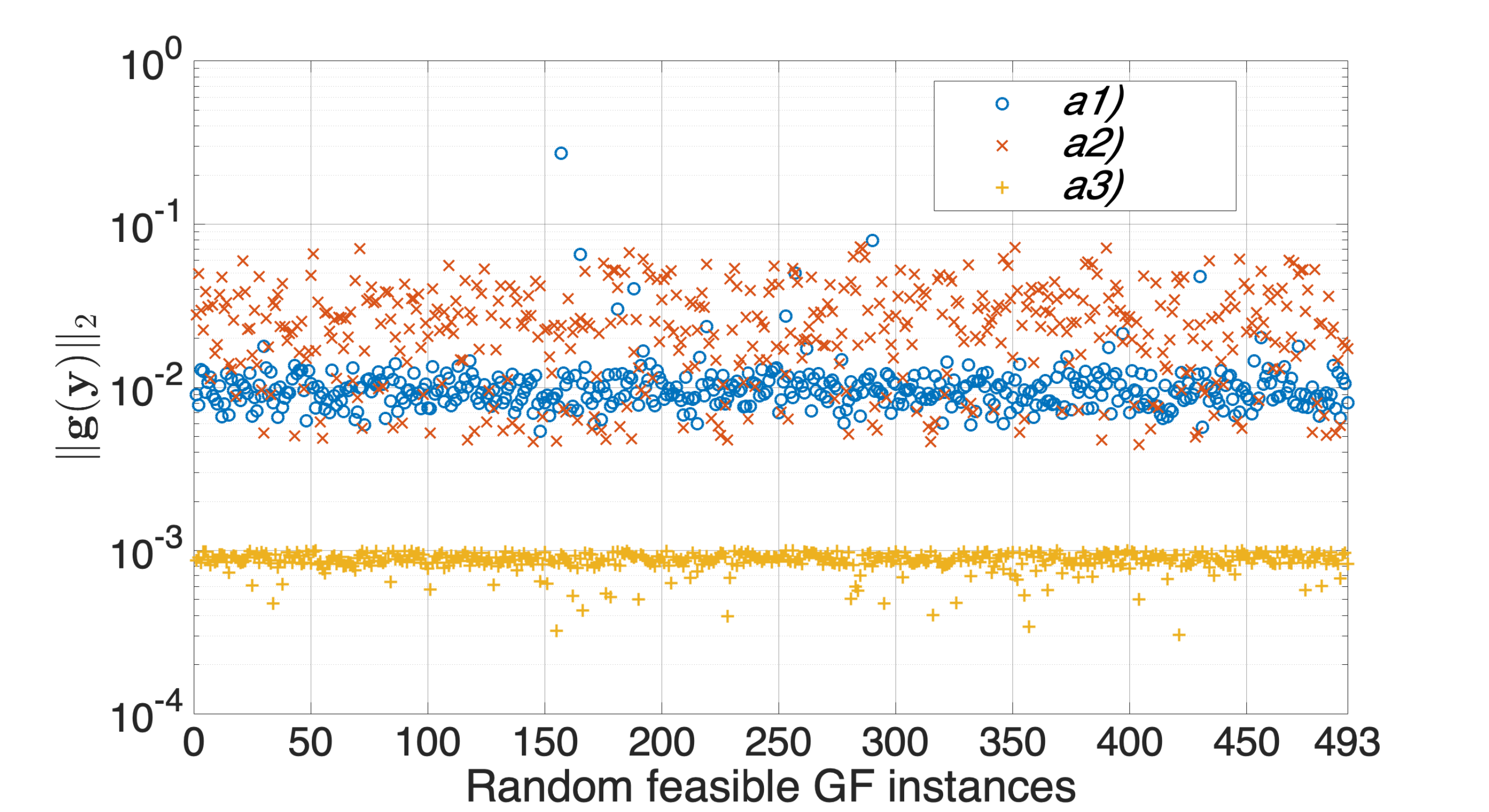}
	\caption{Accuracy measure $\|\bg(\by)\|_2$ for GF solutions obtained by MI-QCQP in \eqref{eq:G2} followed by Algorithm~\ref{algo:1}, and GF solutions found by the Newton Raphson iterations for different initializations.}
	\label{fig:nrcomp}
\end{figure}

A second set of tests were conducted on the GasLib-40 NGN with $500$ random injections and compressor ratios generated as described earlier. The MI-QCQP solver deemed $7$ of the $500$ instances as infeasible. All $500$ instances were then solved with the NR-based solver with flows initialized at $(\bA^\top)^\dagger\bq$ and pressures at $\psi_{1}$, and $\mu$ was set to $\mu=0.9$. A steep decline in $\|\bg(\by_{t})\|_2$ was observed in the first few (roughly 10) iterations, while the tolerance of $10^{-3}$ was not attained within the $50$ iterations limit. However, if the NR solver is initialized at the solution of MI-QCQP and Algorithm~\ref{algo:1}, the convergence criteria of $10^{-3}$ was attained at an average of $7.8$ iterations. The values of $\|\bg(\by)\|_2$ attained by three solution techniques \emph{a1)--a3)} are shown in Fig.~\ref{fig:nrcomp}. The results suggest that the accuracy of the MI-QCQP solver is better than that of \emph{a2)}, which is a prudent initialization. However, if the NR-based solver is warm-started with the solution of MI-QCQP, an order of magnitude improvement in accuracy is observed. On the computational front, the MI-QCQP solver alongside Algorithm~\ref{algo:1} is efficient with median solving time of $1.52$~sec. However, as anticipated, the NR solvers have superior performance with median solving time of $0.17$~sec. Finally, inspecting the $7$ instances deemed infeasible by the MI-QCQP solver, the solution obtained by \emph{a2)} indicates violation of \eqref{seq:compb}; demonstrating the merit of the proposed MI-QCQP towards certifying infeasibility of GF instances. 

\section{Conclusions}\label{sec:conclusions}
Exploiting recent results from graph theory and convex relaxations, this work provides a fresh perspective on the steady-state GF problem. The uniqueness of the GF solution has been established in a generalized setting for arbitrary NGN topologies, multiplicative compressors and multiple fixed-pressure nodes. Granted that the GF solution is unique, constrained and unconstrained versions of convex energy function minimization-based GF solvers have been proposed. These solvers can efficiently solve any GF task instance with a single fixed-pressure node and networks with compressors not on cycles. To expand the scope, an MI-QCQP GF solver had been also proposed relying on a convex relaxation of the Weymouth equation. The relaxation has been shown to be exact under specific network conditions. Numerical tests reveal that the developed MI-QCQP solver succeeds in finding the unique GF solution even when the needed conditions are violated. The success of the MI-QCQP relaxation is attributed to a judiciously designed objective. The developed approach sets forth an analytical platform for ensuring exact relaxation. Evaluating the performance of the developed approach for various optimal gas flow tasks constitutes an interesting research direction.

\appendix
\begin{proof}[Proof of Lemma~\ref{le:1path}]
For an edge $\ell_i\in\mcP_{mn}$, let us name the incident node closer to $m$ as $m_i$, and the other node as $m_{i+1}$, as shown in Fig.~\ref{fig:path}.

		\begin{figure}[h]
		\centering
		\includegraphics[scale=0.6]{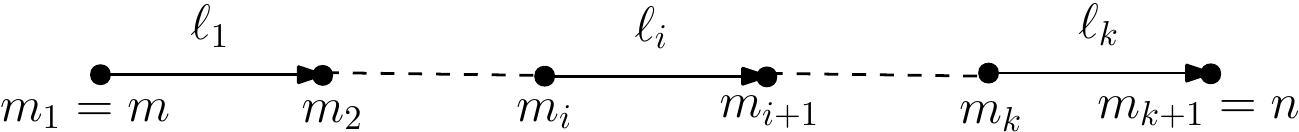}
		\caption{Nomenclature for nodes and edges along $\mcP_{mn}$.}
		\label{fig:path}
	\end{figure}

Let $\bpsi$ and $\bpsi'$ be the pressure vectors corresponding to $\bphi$ and $\bphi'$. Since pressures $\psi_m$ and $\psi_n$ are fixed, it follows $\psi_m=\psi_m'$ and $\psi_n=\psi_n'$. Proving by contradiction, suppose \eqref{eq:le2a} holds. If that is the case, first it will be shown that $\psi_{m_i}'-\psi_{m_{i+1}}'>\psi_{m_i}-\psi_{m_{i+1}}$ for every lossy pipe $\ell_i\in\mcP_{mn}$.
 
Suppose that $\sign(\bphi'-\bphi)\cdot\bpi^{mn}>\bzero$. Let us denote the RHS of \eqref{seq:wey2a} by $w(\phi_{\ell})$. It is evident that $w(\phi_{\ell})$ is monotonically increasing in $\phi_{\ell}$. Hence, for any lossy pipe $\ell_i\in\mcP_{mn}$, it holds
 \begin{align}\label{eq:train}
0&\stackrel{a}{<}\pi^{mn}_{\ell_i}\sign(\phi_{\ell_i}'-\phi_{\ell_i})\notag\\
 &\stackrel{b}{=}\pi^{mn}_{\ell_i}\sign(w(\phi_{\ell_i}')-w(\phi_{\ell_i}))\notag\\
 &\stackrel{c}{=}\sign( \pi^{mn}_{\ell_i})\sign(w(\phi_{\ell_i}')-w(\phi_{\ell_i}))\notag\\
 &\stackrel{d}{=}\sign(\pi^{mn}_{\ell_i}w(\phi_{\ell_i}')-\pi^{mn}_{\ell_i}w(\phi_{\ell_i}))\notag\\
 &\stackrel{e}{=}\sign((\psi_{m_{i}}'-\psi_{m_{i+1}}')-(\psi_{m_{i}}-\psi_{m_{i+1}})),
 \end{align}
 where $(a)$ holds by hypothesis; $(b)$ stems from the monotonicity of $w(\phi_{\ell})$; $(c)$ holds because $\pi^{mn}_{\ell_i}\in\{0,1,-1\}$; $(d)$ holds from the property of $\sign$ by definition; and $(e)$ from the definition of $\bpi^{mn}$ and \eqref{seq:wey2a}. The inequality \eqref{eq:train} implies 
 \begin{equation}\label{eq:lossypipe}
 \psi_{m_i}'-\psi_{m_{i+1}}'>\psi_{m_i}-\psi_{m_{i+1}}.
 \end{equation}
 
 Let us now apply \eqref{eq:lossypipe} and \eqref{seq:compa} for the edges $\ell_1$ to $\ell_k$ along $\mcP_{mn}$. For the fixed pressure node $m$, we have $\psi_m=\psi_m'$. If $\ell_1$ is a lossy pipe, we get $\psi_{m_2}'<\psi_{m_2}$ from \eqref{eq:lossypipe}; otherwise $\psi_{m_2}'=\psi_{m_2}$ from \eqref{seq:compa}. Similarly, we can show that $\psi_{m_3}'\leq\psi_{m_3}$, where the equality holds only if both $\ell_1$ and $\ell_2$ are compressors. However, this is practically impossible as every compressor is modeled as an ideal compressor followed by a lossy pipe, necessitating $\psi_{m_3}'<\psi_{m_3}$. Continuing the process for all edges along $\mcP_{mn}$ yields $\psi_n'<\psi_n$, which contradicts with node $n$ being a fixed-pressure node. Similarly, the assumption $\sign(\bphi'-\bphi)\cdot\bpi^{mn}<0$ leads to a contradiction by yielding $\psi_n'>\psi_n$.
\end{proof}

\begin{proof}[Proof of Lemma~\ref{le:networkflow}]
Given the two pairs $(\bq,\bphi)$ and $(\bq',\bphi')$ satisfying \eqref{eq:mc2} and $\bq\neq\bq'$, let us define $\tbphi:=\bphi'-\bphi$ and $\tbq:=\bq'-\bq$. By applying \eqref{eq:mc2} on $(\bq,\bphi)$ and $(\bq',\bphi')$, and taking the difference, we get
\begin{equation}\label{seq:delqa}
\bA^\top\tbphi=\tbq.
\end{equation}
Since $\bone\in\nullspace(\bA)$, premultiplying \eqref{seq:delqa} by $\bone^\top$ provides
\begin{equation}\label{seq:delqb}
\bone^\top\tbq=\bzero.
\end{equation}
From \eqref{seq:delqa}--\eqref{seq:delqb}, the pair $(\tbq,\tbphi)$ qualifies as a set of balanced gas injections. By definition of $(\tbq,\tbphi)$, proving \eqref{eq:netflow} is equivalent to showing there exists a path $\mcP_{mn}$ for which
\begin{subequations}\label{eq:netflow2}
\begin{align}
&\tbphi\odot\bpi^{mn}>\bzero\label{seq:netflowa2}\\
&\tilde{q}_m>0~~\text{and}~~ \tilde{q}_n<0.\label{seq:netflowb2}
\end{align}
\end{subequations}
To prove the existence of such a path, we use the ensuing result based on~\cite[Th.~8.8]{korte2012combinatorial}.

\begin{lemma}[\cite{korte2012combinatorial}]\label{le:st}
Given a graph with injection $q$ at node $s$, demand $q$ at node $t$, and zero injections at all other nodes, there exists an $s$-$t$ path with flow directions along the path from $s$ to $t$. 
\end{lemma}

Lemma~\ref{le:st} considers a single-source single-destination network flow setup. We transform our problem to this setup through the next steps; see also Fig.~\ref{fig:augmentG}:
\renewcommand{\labelenumi}{\emph{\arabic{enumi})}}
\begin{enumerate}
\item The nodes of graph $\mcG$ are partitioned into the subset with positive $\mcN_+:\{n\in\mcN:\tilde{q}_n>0\}$; negative $\mcN_{-}:\{n\in\mcN:\tilde{q}_n<0\}$; and zero injections $\mcN_0:\{n\in\mcN:\tilde{q}_n=0\}$. Because $\tbq\neq\bzero$, the sets $\mcN_+$ and $\mcN_{-}$ are non-empty.
\item Augment $\mcG$ by adding nodes $s$ and $t$.
\item All nodes in $\mcN_+$ are connected to node $s$, and all nodes in $\mcN_{-}$ are connected to node $t$.
\item The injections in $\mcN_+$ are lumped in node $s$ by setting the flows $\tilde{\phi}_{sn}=\tilde{q}_n$ for all $n\in\mcN_+$. Similarly, the demands in $\mcN_{-}$ are lumped in node $t$ by setting the flows $\tilde{\phi}_{nt}=-\tilde{q}_n$ for all $n\in\mcN_{-}$. 
\end{enumerate}

Applying Lemma~\ref{le:st} on this augmented graph, there exists a path $\mcP_{st}$ with flow directions from $s$ to $t$. For any such path $\mcP_{st}$, eliminate the first and last edges to get a path $\mcP_{mn}$ with $m\in\mcN_+$ and $n\in\mcN_{-}$. Claim~\eqref{seq:netflowb2} follows by construction. We next show \eqref{seq:netflowa2}: For each edge $\ell\in\mcP_{mn}$, it was shown that the direction of $\tilde{\phi}_{\ell}$ is along the path $\mcP_{mn}$. If $\pi_{\ell}^{mn}=+1$, the direction of edge $\ell$ agrees with the direction of $\mcP_{mn}$. Since $\tilde{\phi}_\ell$ is along $\mcP_{mn}$, then $\tilde{\phi}_{\ell}>0$. If $\pi_{\ell}^{mn}=-1$, the direction of edge $\ell$ is opposite to the direction of $\mcP_{mn}$. Since $\tilde{\phi}_\ell$ is along $\mcP_{mn}$, then $\tilde{\phi}_{\ell}<0$. Either way, it holds that $\tilde{\phi}_{\ell}\pi_{\ell}^{mn}>0$  for all $\ell\in\mcP_{mn}$, which proves \eqref{seq:netflowa}.
\end{proof}

\begin{figure}[t]
	\centering
	\includegraphics[scale=0.7]{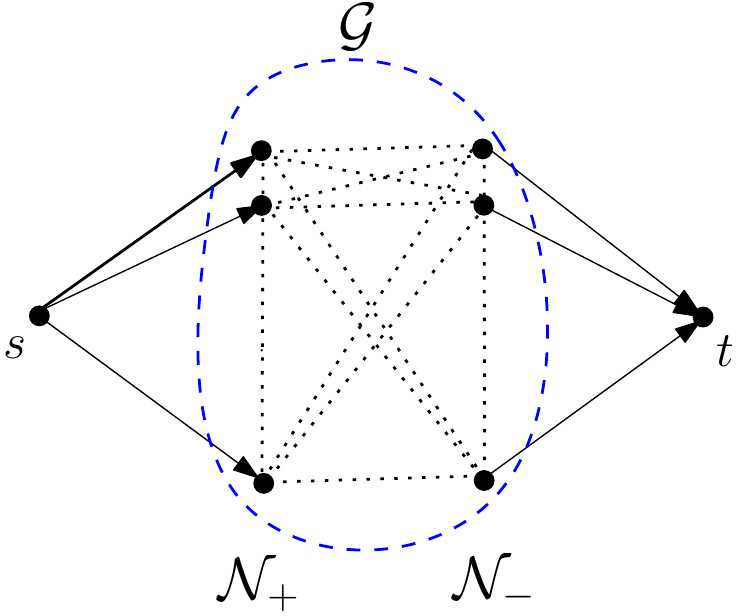}
	\caption{Augmented NGN graph.}
	\label{fig:augmentG}
\end{figure}

\begin{proof}[Proof of Theorem~\ref{th:gf1exact}]
Before proving the main result, we will need two preliminary results.

\begin{lemma}\label{le:ob1}
	For a lossy pipe $\ell=(m,n)$ not on an active cycle, if the triplet $(\psi_m,\psi_n,\phi_\ell)$ satisfies \eqref{eq:weyMC}, then the triplet $(\psi_m+\delta,\psi_n+\delta,\phi_\ell)$ also satisfies \eqref{eq:weyMC} for any finite $\delta$.
\end{lemma}	

Lemma~\ref{le:ob1} follows directly from the fact that \eqref{eq:weyMC} involves pressure differences rather than pressures.

\begin{lemma}\label{le:ob2}
Consider an active cycle $\mcC_0$ and index its nodes as $\{0,\dots,k\}$. Given a fixed pressure $\psi_0$ and flows $\{\phi_{\ell}\}_{\ell\in\mcC_0}$ satisfying Condition~\ref{con:C3} and \eqref{seq:compb}, there exists a set of pressures $\{\psi_i\}_{i=1}^k$ satisfying \eqref{eq:weyMC} and \eqref{seq:compa}.
\end{lemma}	

\begin{proof}
From Condition~\ref{con:C3} and the fact that a compressor is modeled as an ideal compressor followed by a lossy pipe, it is not hard to see that there must exist a node $k\in \mcC_0$ that leads to one of the four flow scenarios shown in Fig.~\ref{fig:scenario}. 

\begin{figure}[t]
	\centering
	\includegraphics[scale=0.65]{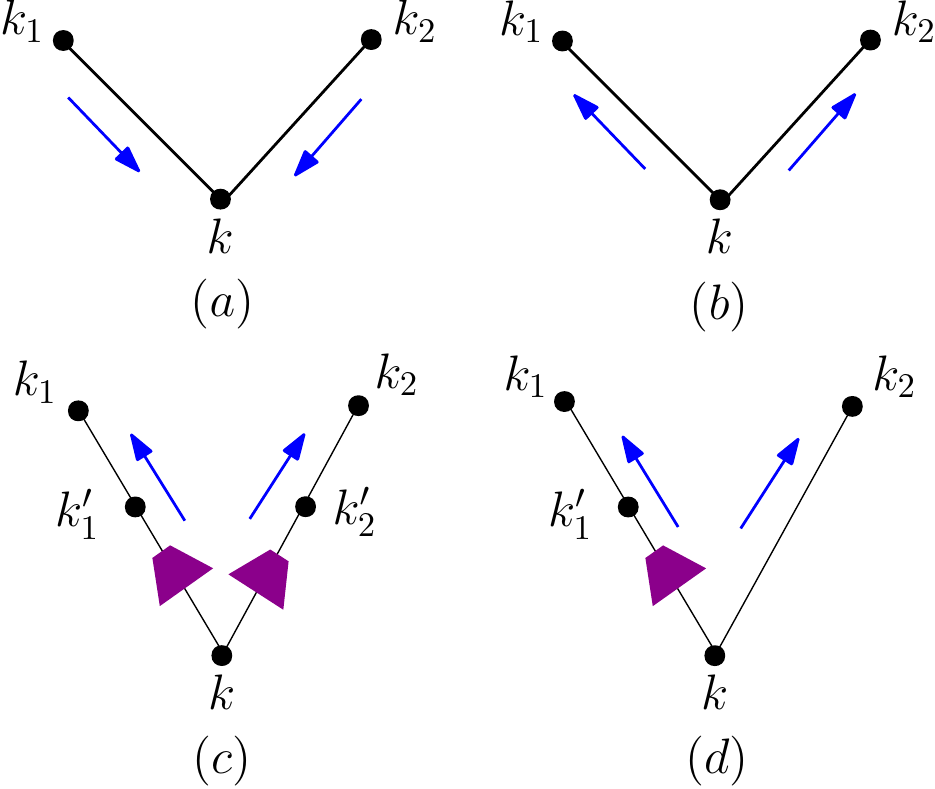}
	\caption{Four possible scenarios for a cycle with non-circulating gas flow. The arrows represent the actual gas flow directions.}
	\label{fig:scenario}
\end{figure}

Proving by construction, we will next define pressures $\{\psi_i\}_{i=1}^k$ such that \eqref{eq:weyMC} and \eqref{seq:compa} are satisfied for all edges in $\mcC_0$. Traversing the paths $0\rightarrow k_1$ and $0\rightarrow k_2$, one can recursively define pressures for all nodes using $\psi_0$ and flows $\{\phi_{\ell}\}_{\ell\in\mcC_0}$ based on the exact Weymouth equation \eqref{eq:wey2} and \eqref{seq:compa}. The pressures on the remaining nodes of $\mcC_0$ can be defined for the four scenarios of Fig.~\ref{fig:scenario} as follows:
\begin{align*}
(a)&~\psi_k:=\min\{\psi_{k_1}-a_{k_1k}\phi_{k_1k}^2,\psi_{k_2}-a_{k_2k}\phi_{k_2k}^2\}\\
(b)&~\psi_k:=\max\{\psi_{k_1}+a_{kk_1}\phi_{kk_1}^2,\psi_{k_2}+a_{kk_2}\phi_{kk_2}^2\}\\
(c)&~\psi_k:=\max\left\{\frac{\psi_{k_1}+a_{k_1'k_1}\phi_{k_1'k_1}^2}{\alpha_{kk_1'}},\frac{\psi_{k_2}+a_{k_2'k_2}\phi_{k_2'k_2}^2}{\alpha_{kk_2'}}\right\}\\
&~ \psi_{k_1'}:=\alpha_{kk_1'}\psi_k,\quad \psi_{k_2'}:=\alpha_{kk_2'}\psi_k\\
(d)&~\psi_k:=\max\left\{\frac{\psi_{k_1}+a_{k_1'k_1}\phi_{k_1'k_1}^2}{\alpha_{kk_1'}},\psi_{k_2}+a_{kk_2}\phi_{kk_2}^2\right\}\\
&~\psi_{k_1'}:=\alpha_{kk_1'}\psi_k.
\end{align*}
To see that the constructed pressures satisfy \eqref{eq:weyMC}, take for example scenario $(a)$. Applying \eqref{eq:weyMC} along the edges $(k_1,k)$ and $(k_2,k)$ yield that $\psi_k$ should satisfy $\psi_k\leq \psi_{k_1}-a_{k_1k}\phi_{k_1k}^2$ and $\psi_k\leq \psi_{k_2}-a_{k_2k}\phi_{k_2k}^2$. This is indeed the case by selecting $\psi_k$ as the minimum of the two RHS. Similar reasoning applies to the other scenarios.
\end{proof}

Proceeding with the proof of Theorem~\ref{th:gf1exact}, let $(\bphi,\bpsi)$ be the unique solution to \eqref{eq:G1}, and $(\bphi',\bpsi')$ a minimizer of \eqref{eq:G2}. Proving by contradiction, assume that there exists an edge $\ell$ not belonging to an active cycle, such that $\phi_\ell'\neq\phi_\ell$. Recall that the set of all active cycles is denoted by $\mcS_{\mcC}^a$. Since both flow vectors satisfy \eqref{eq:mc2}, their difference $\bn:=\bphi-\bphi'$ must lie in the nullspace of $\bA^\top$. The nullspace of $\bA^\top$ is spanned by the indicator vectors for all fundamental cycles in the gas network graph~\cite[Corollary~14.2.3]{GodsilRoyle}. Therefore, the entries of $\bn$ related to edges not on a cycle must be zero. Since by hypothesis $\ell\notin \mcS_\mcC^a$, edge $\ell$ should belong to one of the cycles in $\mcS_\mcC\setminus\mcS_\mcC^a$. This non-active cycle will be henceforth termed $\mcC$. 

The rest of the proof is organized in three parts: Part I constructs a flow vector $\hbphi$ that satisfies \eqref{eq:mc2} and \eqref{seq:compb}. Part II shows there exists a $\hat{\bpsi}$ so that the pair $(\hbphi,\hat{\bpsi})$ is feasible for \eqref{eq:G2}. Part III shows that $(\hbphi,\hat{\bpsi})$ attains a smaller objective for \eqref{eq:G2}, thus contradicting the optimality of $(\bphi',\bpsi')$.

\emph{Part I}: Define the flow vector $\hbphi$ as
\begin{equation}\label{eq:hatphi}
\hat{\phi}_\ell:=
\begin{cases}
\phi_\ell, &\ell\in \mcC\\
\phi_\ell, &\ell ~\text{belongs to any active cycle}\\
\phi_\ell', &\text{otherwise}
\end{cases}.
\end{equation}
By construction, vector $\hbphi$ satisfies
\begin{equation}\label{eq:phidiff}
\bphi'-\hbphi=\lambda\bn^\mcC+\bn^a
\end{equation}
where $\bn^\mcC$ is the indicator vector for cycle $\mcC$; the constant $\lambda$ is nonzero; and vector $\bn^a\in\nullspace(\bA^\top)$ can have nonzero entries only for edges in active cycles. Since $\bphi'$ satisfies constraint \eqref{eq:mc2} and $\bA^\top\bn^\mcC=\bA^\top\bn^a=\bzero$, then $\bA^\top\hbphi=\bA^\top\bphi'=\bq$. This proves that $\hbphi$ satisfies \eqref{eq:mc2}. Note that $\hbphi$ is constructed by selecting entries from $\bphi$ and $\bphi'$. Granted both $\bphi$ and $\bphi'$ satisfy \eqref{seq:compb}, vector $\hbphi$  trivially satisfies \eqref{seq:compb} too.

\emph{Part II}: We will delineate the steps for constructing a vector of pressures $\hbpsi$ such that $(\hbphi,\hat{\bpsi})$ is feasible for \eqref{eq:G2}. Let us select a spanning tree $\mcT$ of the NGN graph $\mcG$ rooted at the reference $r$. We shall define the pressures $\hat{\psi}_n$'s while traversing $\mcT$ via depth-first search. In such a traversal, the following three cases may be identified on arriving at any node $n$: 

\emph{Case 1:} Node $n$ is neither in $\mcC$ nor on an active cycle. Let $n-1$ be the parent node of $n$ in $\mcT$ and define
$$\htpsi_{n}:=\begin{cases}
\alpha_{n-1,n} \htpsi_{n-1} &,~\text{if}~ (n-1,n)\in\mcP_a\\
\htpsi_{n-1} +(\psi_{n}'-\psi_{n-1}') &,~\text{if}~(n-1,n)\in\bmcP_{a}
\end{cases}.$$
 Since the edge $(n-1,n)$ is not in $\mcC\cup\mcS_{\mcC}^a$, we have $\hat{\phi}_{n-1,n}={\phi}_{n-1,n}'$ from \eqref{eq:hatphi}. Therefore, if $(n-1,n)$ is a lossy pipe, Lemma~\ref{le:ob1} ensures that the defined pressure $\htpsi_{n}$ satisfies \eqref{eq:weyMC}. Moreover, if $(n-1,n)$ is a compressor, constraint \eqref{seq:compa} is satisfied trivially by definition.

\emph{Case 2:} Node $n$ is in $\mcC$. If $n$ is the first node in $\mcC$ to be visited, define $\htpsi_{n}$ as in \emph{Case 1}. Then, define the pressures for the remaining nodes $i\in\mcC$ as $\hat{\psi}_{i}:=\psi_{i}+(\htpsi_{n}-\psi_{n})$. Note from \eqref{eq:hatphi} that the flows along $\mcC$ are assigned from $\bphi$, the pair $(\bphi,\bpsi)$ satisfies \eqref{eq:wey2} and hence the relaxed Weymouth \eqref{eq:weyMC} as well. The constructed pressures $\htpsi_{i}$'s for $i\in\mcC$ are simply a shifted version of the pressures $\psi_{i}$'s. Therefore, the pressures $\htpsi_{i}$'s satisfy \eqref{eq:weyMC} from Lemma~\ref{le:ob1}. Mark all nodes in $\mcC$ as traversed and continue.

\emph{Case 3:} Node $n$ is in an active cycle $\mcC_a$. If $n$ is the first node in $\mcC_a$ to be traversed, define the $\htpsi_{n}$ as in \emph{Case 1}. Then, define the pressure for the remaining nodes $i\in\mcC_a$ using Lemma~\ref{le:ob2}. Mark all nodes in $\mcC_a$ as traversed and continue.

Since the constructed pressures satisfy \eqref{eq:weyMC} and \eqref{seq:compa}, the pair $(\hbphi,\hat{\bpsi})$ is feasible for \eqref{eq:G2}. Observe that the pressure drop across lossy pipes not in $\mcC$ is ${\htpsi_m-\htpsi_n=\psi_m'-\psi_n'}$ for \emph{Case 1}; and ${\htpsi_m-\htpsi_n=\psi_m-\psi_n}$ for lossy pipes in $\mcC$ under \emph{Case 2}. This fact is imperative for the ensuing Part III.

\emph{Part III}: We will next show that $r(\bpsi')>r(\hat{\bpsi})$ to contradict the optimality of $\bpsi'$. Note that the objective $r(\bpsi)$ in \eqref{eq:G2} sums up the absolute pressure differences along lossy pipes, but not on active cycles. Since by construction these differences have changed only along $\mcC$, we get 
\begin{equation}\label{eq:costdif}
r(\bpsi')-r(\hat{\bpsi})=\sum_{(m,n)\in \mcC}|\psi_m'-\psi_n'|-|\htpsi_m-\htpsi_n|.
\end{equation}
As the pressure differences depend on flows, we next compare the entries of $\hbphi$ and $\bphi'$ along $\mcC$ using \eqref{eq:phidiff}. Since the edge directions are assigned arbitrarily, assume wlog that $\htphi_{mn}\geq0$ for all $(m,n)\in\mcC$. Given $\bn_\mcC$ and \eqref{eq:phidiff}, one can find the value of $\lambda$. If $\lambda<0$, reverse the reference direction for cycle $\mcC$ to get a positive $\lambda$. Because of this, we can assume $\lambda>0$.

Recall that $\bn_\mcC\in\{0,\pm 1\}^P$. Partition the set of edges in $\mcC$ into mutually exclusive sets $\hat{\mcP}_+$ and $\hat{\mcP}_-$ based on positive and negative entries of $\bn_\mcC$, respectively. From \eqref{eq:phidiff}, it follows
\begin{align}\label{eq:phihatprime}
0\leq&\htphi_\ell<\phi_{\ell}',\quad\forall\ell\in\hat{\mcP}_+.
\end{align}
Summing up the pressure drops along $\mcC$ for $\hat{\bpsi}$ should be zero. Since the pressure drops along $\mcC$ are positive for the edges in $\hat{\mcP}_+$, and negative along the edges in $\hat{\mcP}_-$, it holds that 
\begin{align}\label{eq:sumhat}
&\sum_{(m,n)\in\hat{\mcP}_+}(\hat{\psi}_m-\hat{\psi}_n)=\sum_{(m,n)\in\hat{\mcP}_-}(\hat{\psi}_m-\hat{\psi}_n)\nonumber\\
\implies~&\sum_{(m,n)\in \mcC}|\hat{\psi}_m-\hat{\psi}_n|=2\sum_{(m,n)\in\hat{\mcP}_+}(\hat{\psi}_m-\hat{\psi}_n)
\end{align}
where the absolute value is trivial since $\hat{\phi}_{mn}\geq0$ for all $(m,n)\in \mcC$.

Drawing similar relations on $\bpsi'$, define the set $\mcP_+'\subset \mcC$ containing any edge $(m,n)\in \mcC$ such that the flow $\phi_{mn}'$ is along the direction of $\bn_\mcC$. Using the same argument as in \eqref{eq:sumhat} for $\bpsi'$, we obtain
\begin{align}\label{eq:sumtilde}
\sum_{(m,n)\in \mcC}|\psi_m'-\psi_n'|=2\sum_{(m,n)\in\mcP_+'}(\psi_m'-\psi_n').
\end{align}
Because the flows in $\hbphi$ for the edges in $\hat{\mcP}_+$ are aligned with $\bn_\mcC$  and $\phi_\ell'>\hat{\phi}_\ell$ for these edges from \eqref{eq:phihatprime}, it follows that $\hat{\mcP}_+\subseteq\mcP_+'$. Using the latter in \eqref{eq:sumtilde}, we get
\begin{align}\label{eq:P+less}
2\sum_{(m,n)\in\hat{\mcP}_+}(\psi_m'-\psi_n')
&\leq 2\sum_{(m,n)\in\mcP_+'}(\psi_m'-\psi_n')\nonumber\\
&=\sum_{(m,n)\in \mcC}|\psi_m'-\psi_n'|.
\end{align}

For every edge $\ell=(m,n)\in\hat{\mcP}_+$, it holds that
\begin{equation}\label{eq:P+less2}
\hat{\psi}_m-\hat{\psi}_n\stackrel{(a)}{=}a_\ell\hat{\phi}_\ell^2\stackrel{(b)}{<} a_\ell\phi_\ell^{'2}\stackrel{(c)}{\leq} \psi_m'-\psi_n'
\end{equation}
where $(a)$ comes from the definition of pressures in \emph{Case 2} of Part II; $(b)$ descends from $\phi_\ell'>\hat{\phi}_\ell>0$; and $(c)$ from \eqref{eq:weyMC}. Summing \eqref{eq:P+less2} over all $\ell\in\hat{\mcP}_+$ and multiplying by 2 gives
\begin{align*}
2\sum_{(m,n)\in\hat{\mcP}_+}(\hat{\psi}_m-\hat{\psi}_n)&<2\sum_{(m,n)\in\hat{\mcP}_+}(\psi_m'-\psi_n')\\
\implies~\sum_{(m,n)\in \mcC}|\hat{\psi}_m-\hat{\psi}_n|&<\sum_{(m,n)\in\mcC}|\psi_m'-\psi_n'|
\end{align*}
where the inequality stems from \eqref{eq:sumtilde} and \eqref{eq:P+less}. From \eqref{eq:costdif}, the latter implies that $r(\bpsi')>r(\hat{\bpsi})$, hence contradicting the optimality of $\bpsi'$.
\end{proof}

\balance
\bibliography{myabrv,water,gas}
\bibliographystyle{IEEEtran}

\begin{IEEEbiography}[{\includegraphics[width=1in,height=1.25in,clip,keepaspectratio]{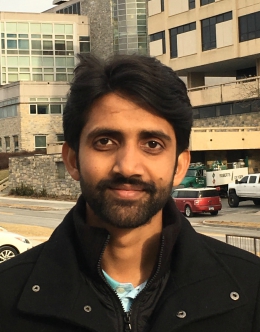}}] {Manish K. Singh} received the B.Tech. degree from the Indian Institute of Technology (BHU), Varanasi, India, in 2013;  and the M.S. degree from Virginia Tech, Blacksburg, VA, USA, in 2018; both in electrical engineering. During 2013-2016, he worked as an Engineer in the Smart Grid Dept. of POWERGRID, the central transmission utility of India. He is currently pursuing a Ph.D. degree at Virginia Tech. His research interests are focused on the application of optimization, control, and graph-theoretic techniques to develop algorithmic solutions for operation and analysis of water, natural gas, and electric power systems.
\end{IEEEbiography} 

\begin{IEEEbiography}[{\includegraphics[width=1in,height=1.25in,clip,keepaspectratio]{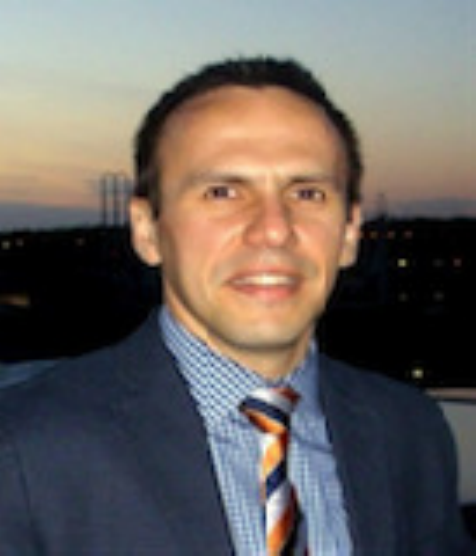}}] {Vassilis Kekatos} (SM'16) is an Assistant Professor with the Bradley Dept. of ECE at Virginia Tech. He obtained his Diploma, M.Sc., and Ph.D. from the Univ. of Patras, Greece, in 2001, 2003, and 2007, respectively. He is a recipient of the NSF Career Award in 2018 and the Marie Curie Fellowship. He has been a research associate with the ECE Dept. at the Univ. of Minnesota, where he received the postdoctoral career development award (honorable mention). During 2014, he stayed with the Univ. of Texas at Austin and the Ohio State Univ. as a visiting researcher. His research focus is on optimization and learning for future energy systems. He is currently serving in the editorial board of the IEEE Trans. on Smart Grid.
\end{IEEEbiography}

\end{document}